\documentclass[12pt]{article}

\usepackage{amssymb,amsmath,amsthm,mathrsfs}

\numberwithin{equation}{section}

\theoremstyle{plain} 
\newtheorem{lemma}{Lemma}[section] 
\newtheorem{theorem}[lemma]{Theorem}  
\newtheorem{corollary}[lemma]{Corollary}

\theoremstyle{definition} 
\newtheorem{definitions}[lemma]{Definitions} 

\theoremstyle{remark}
\newtheorem*{remarks}{Remarks} 

\begin{document}

\begin{center} 
{\bf ON EIGENFUNCTIONS OF THE KERNEL $\frac{1}{2} + \lfloor \frac{1}{xy}\rfloor - \frac{1}{xy}$}   
\end{center} 

\medskip  

\begin{center} 
{\sc N. Watt}   
\end{center}

\vskip 15mm 

\noindent{\bf Abstract.} \ 
The integral kernel $K(x,y) := \frac{1}{2} + \lfloor \frac{1}{x y}\rfloor - \frac{1}{x y}\,$ 
($0<x,y\leq 1$) has connections with the Riemann zeta-function 
and a (recently observed) connection with the Mertens function.
In this paper we begin a general study of the eigenfunctions of $K$.   
Our proofs utilise some classical real analysis (including Lebesgue's theory of integration) 
and elements of the established theory of square integrable symmetric integral kernels. 

\smallskip 


\section{Introduction} 

This paper reports the results of research into the properties of the eigenfunctions of the integral kernel 
$K : [0,1]\times[0,1] \rightarrow {\mathbb R}$ defined by:  
\begin{equation}\label{DefK} 
K(x,y)=
\begin{cases}
\frac{1}{2} -  
\left\{ 
(xy)^{-1} 
\right\} 
&\text{if $0<x,y\leq 1$}, \\  
0 &\text{if $0\leq x,y\leq 1$ and $xy=0$} ,
\end{cases}
\end{equation}
where 
$\{ \alpha\} := \alpha -\lfloor \alpha\rfloor 
=\alpha - \max\{ n\in{\mathbb Z} : n\leq \alpha\}$. 
Our interest in this kernel stems from a connection with Mertens sums $\sum_{n\leq x} \mu(n)$, 
in which $x\geq 1$ and $\mu(n)$ is the M\"obius function. This connection, 
which has its origins in a formula discovered by Mertens himself \cite[Section~3]{Mertens 1897}, 
is not, however, something that shall concern us in this present paper, 
as we have nothing to add to what has already been written about 
it in \cite{Huxley and Watt 2018}, \cite{Watt preprint-a} and \cite{Watt preprint-b}. 
\par 
The kernel $K$ is clearly real and symmetric 
(i.e. one has 
$K(x,y)=K(y,x)\in{\mathbb R}$, for $0\leq x,y\leq 1$). 
It is also a (Lebesgue) measurable function on $[0,1]\times[0,1]$, with Hilbert-Schmidt norm 
\begin{equation}\label{HSnorm}
\left\| K\right\|_{\rm HS} 
:= \left( \int_0^1 \int_0^1 K^2(x,y) dx dy \right)^{1/2} < {\textstyle\frac{1}{2}} , 
\end{equation} 
and satisfies 
\begin{equation}\label{HalfHilbertSchmidt} 
\min\left\{ 0 , \left( \textstyle{-\frac{1}{2}}\right)^p\right\} 
\leq\int_0^1 K^p (x,y) dy < \left( \textstyle{\frac{1}{2}}\right)^p
\quad\text{($p\in\{ 1 , 2\}$, $0\leq x\leq 1$).}
\end{equation}
Note that \eqref{HalfHilbertSchmidt} contains an implicit assertion 
to the effect that, for any constant $a\in [0,1]$,  
the corresponding function $y\mapsto K(a,y)$ 
(and so also the function $y\mapsto K(y,a)$) is measurable on $[0,1]$. 
\par 
In addition to the above mentioned properties, 
$K$ has the property of being {\it non-null} (i.e. one has $\| K\|_{\rm HS} >0$). 
Partly in consequence of this, there exists a maximal orthonormal system 
$\left\{\phi_1,\phi_2, \ldots\right\}\subset L^2 \bigl( [0,1]\bigr)$ such that 
\begin{equation}\label{DefEigenfunction}
\phi_j(x) = \lambda_j \int_0^1 K(x,y) \phi_j(y) dy 
\quad\text{($0\leq x\leq 1$, $j\in{\mathbb N}$),}
\end{equation}
where $\lambda_1,\lambda_2,\ldots $ are certain non-zero real constants: 
for proof, see the discussion of \cite[Section~3.8]{Tricomi 1957} and our remarks 
at the end of this paragraph, and after the next paragraph. 
Following \cite{Tricomi 1957}, we say that the numbers 
$\lambda_1,\lambda_2,\ldots $ are the {\it eigenvalues} of $K$: the associated {\it eigenfunctions} are   
$\phi_1(x),\phi_2(x),\ldots $, respectively. 
In \cite{Watt preprint-a}, we have shown that $K$ has infinitely many distinct positive eigenvalues and 
infinitely many distinct negative eigenvalues. 
\par 
Note that $L^2\bigl( [0,1]\bigr)$ denotes here the space of 
(Lebesgue) measurable functions $f : [0,1] \rightarrow {\mathbb R}$ that are {\it square-integrable} 
(in that $f^2$ is Lebesgue integrable on $[0,1]$), 
and that what is meant (above) by  
{\it orthonormality} is orthonormality with respect 
to the (semi-definite) inner product 
\begin{equation}
\label{DefInnerProduct} 
\left\langle f , g\right\rangle 
:= \int_0^1 f(x) g(y) dx 
\quad\text{($f,g\in L^2\bigl( [0,1]\bigr)$).}
\end{equation} 
Each $f\in L^2 \left( [0,1]\right)$ has  
{\it norm} $\| f\| :=\sqrt{\langle f , f\rangle}$. 
This `norm' is actually only a seminorm on $L^2 \left( [0,1]\right)$, since the condition 
$\| f\| =0$ implies only that $f(x)=0$
{\it almost everywhere} (with respect to the Lebesgue measure). 
\par 
In the theory developed in \cite{Tricomi 1957} it is implicit that our condition \eqref{DefEigenfunction} is  
replaced by the weaker condition that, for $j\in{\mathbb N}$, 
one has $\lambda_j\int_0^1 K(x,y) \phi_j(y) dy = \phi_j(x)$ almost everywhere in $[0,1]$. 
We can justify the stronger condition \eqref{DefEigenfunction} by observing that if   
$\lambda\in{\mathbb R}\backslash\{ 0\}$ and $\phi\in L^2 \left( [0,1]\right)$   
are such that $\lambda\int_0^1 K(x,y) \phi(y) dy = \phi(x)$ almost everywhere in $[0,1]$,   
then, given that we have \eqref{HalfHilbertSchmidt}, it follows by the Cauchy-Schwarz inequality 
that the function $\phi^{\dagger} (x) := \lambda\int_0^1 K(x,y) \phi (y) dy$ 
is an element of $L^2 \left( [0,1]\right)$ that satisfies 
$\lambda\int_0^1 K(x,y) \phi^{\dagger} (y) dy = \phi^{\dagger} (x)$ for all $x\in [0,1]$, 
and has $\| \phi^{\dagger} - \phi \| = 0$, so that $\langle \phi^{\dagger} , \psi\rangle 
= \langle \phi , \psi\rangle$ for all $\psi\in L^2 \left( [0,1]\right)$. 
\par 
In light of what has just been noted (in the last paragraph), we  
make it our convention that a function $\phi$  
be considered an eigenfunction of $K$ 
if and only if it is an element of $L^2 \bigl( [0,1]\bigr)$ that has norm 
$\| \phi\| > 0$ and is such that, for some $\lambda\in{\mathbb R}$ (necessarily 
an eigenvalue of $K$),  
one has $\phi(x) = \lambda \int_0^1 K(x,y) \phi(y) dx$ for all $x\in [0,1]$.
\par
It is shown in \cite[Chapter~2]{Tricomi 1957} that, for kernels 
such as $K$, each eigenvalue $\lambda$ has an {\it index},  
$i(\lambda) := |\{ j\in{\mathbb N} : \lambda_j = \lambda\}|$,  
that is finite. Thus we may follow \cite[Section~3.8~(12)]{Tricomi 1957} in assuming the eigenvalues of $K$ to be numbered in such a way that 
\begin{equation}\label{EigenvalueOrder1} 
0<\left| \lambda_1\right| \leq 
\left| \lambda_2\right| \leq 
\left| \lambda_3\right| \leq \ldots 
\end{equation} 
and  
\begin{equation}\label{EigenvalueOrder2} 
\lambda_j\geq\lambda_{j+1} 
\quad\text{when $|\lambda_j|$ and $|\lambda_{j+1}|$ are equal.} 
\end{equation} 
With this last assumption the sequence $\lambda_1,\lambda_2,\ldots $  becomes uniquely determined: the same cannot be said of the corresponding orthonormal 
sequence of eigenfunctions, 
$\phi_1,\phi_2,\ldots\ $, since one can always substitute $-\phi_j(x)$ in place of $\phi_j(x)$ 
(while other substitutions become possible in the event of having $i(\lambda_j) \geq 2$). 
\par 
Aside from the connection with Mertens sums (mentioned in the first paragraph of this section), 
another reason for studying the eigenfunctions of $K$ 
is that there is a connection between this kernel and Riemann's zeta-function, $\zeta(s)$. 
In order to make this connection apparent we begin by observing that, if $f$ is a continuous 
real valued function on $[0,1]$ that satisfies 
\begin{equation*} 
\int_0^1 \frac{|f(x)| dx}{x} < \infty\;,
\end{equation*} 
then, by application of the most rudimentary form of the Euler-Maclaurin summation formula 
\cite[Theorem~7.13]{Apostol 1974}, it may be established that when $0<x\leq 1$ one has: 
\begin{equation}\label{K_action=EuMacErr} 
\int_0^1 K(x,y) f(y) dy 
= \sum_{n>\frac1x} F\left(\frac{1}{n x}\right) 
- \int_{\frac1x}^{\infty} F\left(\frac{1}{\nu x}\right) d\nu 
+ F(1) K(1,x) \;, 
\end{equation} 
where $F(z):=\int_0^z f(y) dy\,$ ($0\leq z\leq 1$). 
In particular, when $f(x):= x^s\,$ ($0\leq x\leq 1$) and $s$ is 
any complex constant satisfying ${\rm Re}(s) >0$, one finds (by \eqref{K_action=EuMacErr}) 
that 
\begin{equation}\label{ZetaConnect-1} 
(s+1)x^{s+1} \int_0^1 K(x,y) y^s dy 
=\zeta(s+1) - \sum_{n\leq\frac1x}\frac{1}{n^{s+1}} - \frac{x^s}{s} 
+ x^{s+1} K(1,x) 
\end{equation} 
for $0<x\leq 1$. The novelty here lies in the presentation (not the content) 
of this result: see for example \cite[Equation~(3.5.3)]{Titchmarsh 1986}, which is 
equivalent to \eqref{ZetaConnect-1} in the special case where $1/x\in{\mathbb N}$. 
Similarly to what is observed in \cite[Section~3.5]{Titchmarsh 1986}, 
one may deduce, by analytic continuation from the half plane ${\rm Re}(s)>0$, 
that \eqref{ZetaConnect-1} holds for all $s\in{\mathbb C} -\{ 0\}$ 
satisfying the condition ${\rm Re}(s) > -1$. 
\par
Though it is somewhat peripheral to our present discussion, we 
remark that, since it is known that $\zeta(1+s) = s^{-1} + \gamma + O\left( |s|\right)$ for $|s|\leq 1\,$ 
(where $\gamma = 0{\cdot}5772\ldots\ $ is Euler's constant), 
one may deduce from \eqref{ZetaConnect-1} that 
\begin{equation*} 
x \int_0^1 K(x,y) y^0 dy 
=\gamma - \sum_{n\leq\frac1x}\frac{1}{n} + \log\left(\frac1x\right) + x K(1,x) 
\qquad\text{($0<x\leq 1$)} . 
\end{equation*} 
Given that $\zeta(0)=-\frac12$ and $\zeta'(0)=-\frac12 \log(2\pi)$, 
one may (similarly) deduce from \eqref{ZetaConnect-1} and \eqref{DefK} that 
\begin{multline}\label{Stirling_Alt} 
\lim_{s\rightarrow (-1)+} \int_0^1 K(x,y) y^s dy \\ 
=\log\left( \left\lfloor 1/x \right\rfloor !\right) 
- \left\lfloor 1/x\right\rfloor \log\left(1/x\right) 
+ (1/x) - \log\sqrt{2\pi / x}  
\end{multline} 
for $0<x\leq 1$. A well-known result closely related to this is Stirling's formula 
\cite[Equations~(B.25) and~(B.26)]{Montgomery and Vaughan 2007}.   
With the help of Stirling's formula one can show that 
\eqref{Stirling_Alt} would remain valid if the limit that occurs on its left-hand side  
were to be replaced with the improper Riemann integral 
$\lim_{\varepsilon\rightarrow 0+} \int_{\varepsilon}^1 K(x,y) y^{-1} dy$.  
\par 
An alternative way to connect $K(x,y)$ with $\zeta(s)$ begins with the observation  
in \cite[Section~1]{Watt preprint-b} to the effect that if $f$ is a measurable 
complex valued function defined on $[0,1]$ that satisfies $\int_0^1 \left| f(y)\right|^2 dy <\infty$, 
and if one puts 
\begin{equation*} 
g(x):=\int_0^1 K(x,y)f(y) dy\quad\text{for $0\leq x\leq 1$} , 
\end{equation*} 
while taking $F$, $G$ and $h$ to be the functions on $[0,\infty)$ satisfying 
\begin{equation*} 
\sqrt{x}\cdot \left( f(x) , g(x) , K( 1 , x)\right) = \left( F(v) , G(v) , h(v)\right) \in{\mathbb C}^3 
\qquad\text{($0<x=e^{-v}\leq 1$)} ,  
\end{equation*} 
then one will have both 
\begin{equation} \label{Hankel_Op}
G(u) = \int_0^{\infty} h(u+v) F(v) dv =  \left( \Gamma_h F\right) (u) \quad\text{(say)} ,  
\end{equation} 
for $0\leq u <\infty$, and $\int_0^\infty \left| F(v)\right|^2 dv = \int_0^1 \left| f(y)\right|^2 dy <\infty$.  
Note that \eqref{Hankel_Op} implicitly defines $\Gamma_h$ to be 
a certain Hankel operator on the space of complex valued functions that 
are square integrable on $[0,\infty)$.  Researchers investigating   
such operators have found it useful to 
consider the Laplace transform of the relevant kernel function:  
see, for example \cite[Chapter~4]{Partington 1988}. 
In our case the relevant kernel function is $h$. A connection with 
$\zeta(s)$ therefore arises due to our having:  
\begin{align*} 
\left( {\mathcal L} h\right) \left( s - {\textstyle\frac12}\right) 
 &:= \int_0^{\infty}  h(v) e^{-\left( s - \frac12\right) v} dv \\ 
 &=\int_0^{\infty} K\left( 1 , e^{-v}\right) e^{- s v}  dv \\ 
 &=\int_0^1 K(1,y) y^{s - 1} dy =\frac{\zeta(s) - \frac{1}{s-1} - \frac12}{s} =\frac{\zeta(s) - \zeta(0)}{s} - \frac{1}{s-1}  
\end{align*} 
for ${\rm Re}(s) > 0$ 
(the penultimate equality here following by virtue of \eqref{ZetaConnect-1}, with 
$s-1$ substituted for $s$). 
\par 
A third indication of a connection between $K(x,y)$ and $\zeta(s)$ is 
implicit in \cite[Equations~(36), (37) and~(41)]{Huxley and Watt 2018}. 
This connection, and the other two (discussed above) are all  
closely linked: they share a common origin. 
\par
The connections just noted between $K(x,y)$ and $\zeta(s)$ 
play no part in the remainder of this paper, but we do have some hope that 
a worthwhile application of one of them may eventually be found: 
it might (for example) be the case that interesting results 
concerning the eigenfunctions of $K(x,y)$ can be deduced from known properties 
of $\zeta(s)$. 
\par 
In this paper we employ only methods from classical real analysis 
(including some of Lebesgue's theory of integration) together with certain elements 
of the general theory of square integrable symmetric integral kernels 
(our primary reference for this theory being \cite{Tricomi 1957}). 
We have aimed to answer some 
basic questions concerning the eigenfunctions of $K$. 
We shall show, for example, that the eigenfunctions of $K$ 
are continuous on $[0,1]$: this is Theorem~2.10. 
In Theorem~4.1 we find that the eigenfunctions of 
$K$ are differentiable at any point $x\in (0,1)$ 
that is not the reciprocal of a positive integer.  
That theorem also supplies a useful formula for the first derivative of any eigenfunction.  
In Theorem~4.13 we show, in effect, that if $\phi$ is an eigenfunction of $K$, 
then the function $x\mapsto x \phi' (x) + \frac12 \phi(x)$ is a solution of 
a particular integral equation with kernel $K(x,y)$. 
In the latter part of Section~4, we obtain 
(via a well-known theorem of Hilbert and Schmidt) certain corollaries of Theorem~4.13: 
these corollaries have interesting further consequences, which we intend to discuss   
in another paper (currently in preparation). 
\par 
In Section~5 (the final section of the paper) we show that the behaviour of 
any eigenfunction of $K$ approximates that of a certain very simple 
oscillatory function on any neighbourhood $[0,\varepsilon)$ of the point $x=0$ that is sufficiently 
small (in terms of the relavant eigenvalue). Our main results there are Theorems~5.4 and~5.11. 
\par 
In addition to the above mentioned results, 
we also obtain a number of upper bounds for the `sizes' of eigenfunctions and their first 
derivatives: see, in particular, \eqref{PhiBoundedUniformly} and Theorems~3.2, 4.6, 4.9 and~4.11. 
We think it likely that, with more work, and some new ideas, it should 
be possible to significantly improve upon all of these bounds 
(and, as a consequence, improve upon Corollary~4.10 also).  
\par 
In Lemmas~2.3, 2.4, 5.3, 5.5 and~5.6, Theorems~2.7, 2.9 and~2.11, and 
Corollaries~2.8 and~2.12, we obtain certain results concerning 
the iterated kernel $K_2 (x,y)$ defined at the start of the next section. 
Most of these results are required for use in other proofs, but 
some were included in this paper due to their own intrinsic interest. 
The function $K_2 (x,y)$ is, in our opinion, interesting enough to merit further study:  
our Remarks following the proof of Lemma~2.4 are connected with this matter. 

\section{Continuity} 

\begin{definitions} 
Following \cite{Tricomi 1957}, we define   
\begin{equation}\label{K2symmetrically} 
K_2 (x,y) := 
\int_0^1 K(x,z) K(z,y) dz  = 
\int_0^1 K(x,z) K(y,z) dz , 
\end{equation} 
for $0\leq x,y\leq 1$. 
\end{definitions} 
\par 
Like $K$, the function $K_2$ is real-valued, measurable and square-integrable on $[0,1]\times [0,1]$. 
The final equality in \eqref{K2symmetrically} 
holds by virtue of $K$ being symmetric: 
we deduce from it that $K_2$ is a symmetric integral kernel. 
\par
We shall need to make use of the fact that any eigenfunction of $K$ is also an 
eigenfunction of $K_2$. In particular, when $\phi$ is an eigenfunction of $K$, 
and $\lambda$ the associated eigenvalue, one has:
\begin{equation}\label{K2Eigenfunction} 
\phi (x) = \lambda^2 \int_0^1 K_2 (x,y) \phi (y) dy 
\qquad\text{($0\leq x\leq 1$).} 
\end{equation} 
To verify this, observe that, since  
that $K$ is both measurable and bounded on $[0,1]\times [0,1]$, 
while $\phi$ is an element of $L^2\bigl( [0,1]\bigr)$  
satisfying $\lambda \int_0^1 K(z,y) \phi (y) dy = \phi(z)$ for $0\leq z\leq 1$, 
it therefore follows by \eqref{HalfHilbertSchmidt} and  Fubini's theorem that, for $0\leq x\leq 1$, one has 
\begin{align*} 
\phi (x) &= \lambda \int_0^1 K(x,z)\left( \lambda \int_0^1 K(z,y) \phi (y) dy\right) dz \\ 
 &= \lambda^2 \int_0^1 \left( \int_0^1 K(x,z) K(z,y)dz\right) \phi (y) dy , 
\end{align*}
and so (see the definition \eqref{K2symmetrically}) 
the result \eqref{K2Eigenfunction} is obtained. 
\par 
In preparation for our first application of \eqref{K2Eigenfunction}, 
which comes in the proof of Theorem~2.10 (below), we work on adding to what we know about $K_2$. 
\begin{definitions} 
For $n\in{\mathbb N}$ we define $\widetilde B_n (t)$, 
the {\it $n$-th periodic Bernouilli function}, by: 
\begin{equation*} 
\widetilde B_n (t) := B_n\left( \{ t\}\right) 
\quad\text{($t\in{\mathbb R}$),}
\end{equation*} 
where $B_n (x)$ is the Bernouilli polynomial of degree $n$ 
(the definition of which may be found in \cite[Section~24.2]{Olver et al. 2010}). In particular, 
\begin{equation}\label{DefTildeB1andB2}
\widetilde B_1 (t) := \{ t\} - {\textstyle\frac{1}{2}}\quad\text{and}\quad\widetilde B_2 (t) := \{ t\}^2 -\{ t\} + {\textstyle\frac{1}{6}} 
\quad\text{($t\in{\mathbb R}$),} 
\end{equation} 
and so (given \eqref{DefK}) we have: 
\begin{equation}\label{KtoTildeB1} 
\widetilde B_1 \left( \frac{1}{w}\right) = 
- K(x,y) 
\quad\text{($0<x,y\leq 1$ and $xy=w$).} 
\end{equation} 
\end{definitions}  
\begin{lemma} 
For $0<x,y\leq 1$, one has 
\begin{align*} 
K_2 (x,y) &= 
- {\textstyle\frac{1}{2}} x \widetilde B_2 \left( \frac{1}{x}\right) \widetilde B_1 \left( \frac{1}{y}\right)  
+\frac{1}{x} \int_{\frac{1}{x}}^{\infty} \widetilde B_2 (t) \widetilde B_1 \left( \frac{x t}{y}\right) \frac{dt}{t^3} \\ 
&\phantom{{=}} - 
\frac{1}{2y} \int_{\frac{1}{x}}^{\infty} \widetilde B_2 (t) \frac{dt}{t^2} + 
\frac{x}{2y^2} 
\sum_{m > \frac{1}{y}} \frac{\widetilde B_2 \left( \frac{my}{x}\right)}{m^2} . 
\end{align*} 
\end{lemma} 
\begin{proof} 
Let $0< x,y\leq 1$. By \eqref{K2symmetrically} and \eqref{KtoTildeB1}, 
\begin{equation*} 
K_2(x,y) = \int_{0+}^1 \widetilde B_1 \left( \frac{1}{xz}\right) \widetilde B_1 \left( \frac{1}{yz}\right) dz = 
\frac{1}{x}\int_{\frac{1}{x}}^{\infty} \widetilde B_1 (t) \widetilde B_1 \left( \frac{xt}{y}\right) \frac{dt}{t^2} . 
\end{equation*} 
Since $\int_a^b \widetilde B_1 (t) dt = \frac{1}{2}\widetilde B_2 (b) - \frac{1}{2}\widetilde B_2 (a)$ 
for $a,b\in{\mathbb R}$, it follows from the above equations that 
\begin{align*} 
K_2(x,y) &= 
\frac{1}{2x}\int_{\left( \frac{1}{x}\right){\scriptscriptstyle +}}^{\infty} 
t^{-2} \widetilde B_1 \left( \frac{xt}{y}\right) d\widetilde B_2 (t) \\ 
 &= \frac{1}{2x}\left( 
\left[ t^{-2} \widetilde B_1 \left( \frac{xt}{y}\right) \widetilde B_2 (t)\right]_{\left( \frac{1}{x}\right){\scriptscriptstyle +}}^{\infty} - 
\int_{\left( \frac{1}{x}\right){\scriptscriptstyle +}}^{\infty} 
\widetilde B_2 (t) d\left( t^{-2} \widetilde B_1 \left( \frac{xt}{y}\right)\right)\right) 
\end{align*} 
(the latter equality being obtained through integration by parts). 
By \eqref{DefTildeB1andB2} we have here  
$t^{-2} \widetilde B_1 \left( \frac{xt}{y}\right) \widetilde B_2 (t) \rightarrow 0$ as 
$t\rightarrow\infty$; since the function $t\mapsto\{ t\}$ is right-continuous, we have also 
$t^{-2} \widetilde B_1 \left( \frac{xt}{y}\right) \widetilde B_2 (t) \rightarrow 
x^2 \widetilde B_1 \left( \frac{1}{y}\right) \widetilde B_2 \left( \frac{1}{x}\right)$ as 
$t\rightarrow \left( \frac{1}{x}\right){\scriptstyle +}$. We have, moreover, 
\begin{multline*} 
\int_{\left( \frac{1}{x}\right){\scriptscriptstyle +}}^{\infty} 
\widetilde B_2 (t) d\left( t^{-2} \widetilde B_1 \left( \frac{xt}{y}\right)\right) \\ 
= \int_{\left( \frac{1}{x}\right){\scriptscriptstyle +}}^{\infty} 
\widetilde B_2 (t) \widetilde B_1 \left( \frac{xt}{y}\right) d\left( t^{-2} \right) + 
\int_{\left( \frac{1}{x}\right){\scriptscriptstyle +}}^{\infty} 
\widetilde B_2 (t) t^{-2} d\left(  \widetilde B_1 \left( \frac{xt}{y}\right)\right) \\  
= -2 \int_{\frac{1}{x}}^{\infty} 
\widetilde B_2 (t) \widetilde B_1 \left( \frac{xt}{y}\right) t^{-3} dt + 
\int_{\left( \frac{1}{x}\right){\scriptscriptstyle +}}^{\infty} 
\widetilde B_2 (t) t^{-2} d\left\{ \frac{xt}{y}\right\}    
\end{multline*} 
and 
\begin{align*} 
\int_{\left( \frac{1}{x}\right){\scriptscriptstyle +}}^{\infty} 
\widetilde B_2 (t) t^{-2} d\left\{\frac{xt}{y}\right\}  
 &= \int_{\frac{1}{x}}^{\infty} 
\widetilde B_2 (t) t^{-2} d\left( \frac{xt}{y}\right) - 
\int_{\left( \frac{1}{x}\right){\scriptscriptstyle +}}^{\infty} 
\widetilde B_2 (t) t^{-2} d\left\lfloor \frac{xt}{y}\right\rfloor \\ 
 &=\frac{x}{y} \int_{\frac{1}{x}}^{\infty} \widetilde B_2 (t) t^{-2} dt - 
 \sum_{m > \frac{1}{y}} \widetilde B_2 \left( \frac{ym}{x}\right) \left( \frac{ym}{x}\right)^{-2} ,
\end{align*} 
and so we obtain what is stated in the lemma
\end{proof} 
\begin{lemma} 
When $0<x,y\leq 1$, one has: 
\begin{equation}\label{K2OverxBound}  
\frac{\left| K_2 (x,y)\right|}{x} \leq 
\frac{1}{12} + 
\frac{x}{\left( 36 \sqrt{3}\right) y} + 
\frac{1}{2 y^2} \Biggl| \sum_{m > \frac{1}{y}} 
\frac{\widetilde B_2 \left( \frac{my}{x}\right)}{m^2} 
\Biggr| , 
\end{equation} 
\begin{equation}\label{K2BoundOffDiagonal}  
\left| K_2 (x,y)\right| \leq 
\left({\textstyle\frac{1}{4} + \frac{1}{36 \sqrt{3}}}\right) 
\cdot 
\frac{\min\{ x , y\}}{\max\{ x , y\}}  
\end{equation} 
and 
\begin{equation} \label{K2BoundOnDiagonal} 
\left| K_2 (x,x) - {\textstyle\frac{1}{12}} \right| \leq 
\left({\textstyle\frac{1}{6} + \frac{1}{36 \sqrt{3}}}\right)\cdot x . 
\end{equation}
\end{lemma} 
\begin{proof} 
The result \eqref{K2OverxBound} follows from Lemma~2.3 by applying the triangle inequality and 
then observing that one has: 
\begin{equation*} 
\left| \widetilde B_2 \left( \frac{1}{x}\right) \widetilde B_1 \left( \frac{1}{y}\right)\right| 
\leq {\textstyle\left(\frac{1}{6}\right)\left(\frac{1}{2}\right) = \frac{1}{12}} , 
\end{equation*} 
\begin{equation*} 
\biggl|\int_{\frac{1}{x}}^{\infty} \widetilde B_2 (t) 
\widetilde B_1 \left( \frac{x t}{y}\right) t^{-3} dt \biggr| 
\leq\int_{\frac{1}{x}}^{\infty} {\textstyle\frac{1}{12}} t^{-3} dt = {\textstyle\frac{1}{24}} x^2 
\end{equation*} 
and  
\begin{align*} 
\int_{\frac{1}{x}}^{\infty} \widetilde B_2 (t) \frac{dt}{t^2} 
&= {\textstyle\frac{1}{3}} \int_{\frac{1}{x}}^{\infty} t^{-2} d\widetilde B_3 (t) \\
&= {\textstyle\frac{1}{3}} \left( 
\left[ t^{-2} B_3 (t)\right]_{\frac{1}{x}}^{\infty} 
- \int_{\frac{1}{x}}^{\infty} B_3 (t) d\left( t^{-2}\right) \right) , 
\end{align*} 
where $B_3(t) = \{ t\}^3 - \frac{3}{2} \{ t\}^2 + \frac{1}{2} \{ t\}$, so that 
$\max_{t\in{\mathbb R}} \left| B_3(t)\right| = \frac{1}{12\sqrt{3}}$ and  
\begin{equation*} 
\biggl| \int_{\frac{1}{x}}^{\infty} B_2 (t) \frac{dt}{t^2} \biggr| 
\leq {\textstyle\frac{1}{36\sqrt{3}}} \left( \left(\frac{1}{x}\right)^{-2} 
+ \int_{\infty}^{\frac{1}{x}} d\left( t^{-2}\right) \right) = {\textstyle\frac{1}{18\sqrt{3}}} x^2 .
\end{equation*} 
\par 
We consider next \eqref{K2BoundOffDiagonal}. Since both sides of this result are 
invariant under the permutation $(x,y)\mapsto (y,x)$, we may assume (in our proof of it) that 
$0<x\leq y\leq 1$. By \eqref{K2OverxBound} and the uniform bound $\left| \widetilde B_2 (t)\right| \leq\frac{1}{6}$, 
we find that 
\begin{align*} 
\left| K_2 (x,y)\right|\cdot \frac{y}{x} &\leq 
\frac{y}{12} + \frac{x}{\left( 36 \sqrt{3}\right)} + 
\frac{1}{12 y} \sum_{m > \frac{1}{y}} \frac{1}{m^2} \\ 
 & \leq \frac{1}{12} + \frac{1}{36 \sqrt{3}} + 
 \frac{1}{12 y} \cdot\left( y^2 + y\right) =  \frac{2 + y}{12} + \frac{1}{36 \sqrt{3}}. 
\end{align*} 
The required case ($0<x\leq y\leq 1$) of \eqref{K2BoundOffDiagonal} follows. 
\par 
In order to obtain \eqref{K2BoundOnDiagonal} (and so complete the proof of the corollary) 
we note firstly that our proof of \eqref{K2OverxBound} shows, in fact, that one has 
\begin{equation*} 
\biggl| \frac{K_2 (x,y)}{x} - \frac{1}{2 y^2} \sum_{m > \frac{1}{y}} 
\frac{\widetilde B_2 \left( \frac{my}{x}\right)}{m^2} \biggr| \leq 
\frac{1}{12} + \frac{x}{\left( 36 \sqrt{3}\right) y} 
\quad\text{($0<x,y\leq 1$).}
\end{equation*} 
By specialising this to the case in which $y=x\in (0,1]$, and then noting that 
\begin{equation*} 
\sum_{m>\frac{1}{x}} \frac{\widetilde B_2 (m)}{m^2} = \sum_{m>\frac{1}{x}} {\textstyle\frac{1}{6}} m^{-2} 
\in \left[ {\textstyle\frac{1}{6}} \left( x - x^2 \right) , 
{\textstyle\frac{1}{6}} \left( x + x^2 \right) \right] , 
\end{equation*} 
one arrives at the bound
$\left| K_2(x,x) - \frac{1}{12}\right| /x \leq \frac{1}{6} + \frac{1}{36 \sqrt{3}}$. 
The result \eqref{K2BoundOnDiagonal} follows. 
\end{proof} 
\begin{remarks} With regard to the above estimate \eqref{K2BoundOnDiagonal}, it should 
be noted that, by a method entirely different from 
the methods used in the proofs of Lemmas~2.3 and~2.4, it can be shown that one has  
\begin{equation}\label{K2xxactly} 
K_2 (x,x) = K^2 (1,x) 
 + \frac{\left\lfloor \frac1x\right\rfloor \log\left(\frac1x\right) 
 - \frac1x + \log\sqrt{\frac{2\pi}{x}} 
 - \log\left( \left\lfloor \frac1x \right\rfloor !\right)}{\left( \frac{x}{2}\right)} \;,    
\end{equation} 
for $0<x\leq 1$. 
We believe that the same method will also yield an interesting formula 
for $K_2 (x,\alpha x)$ in the more general case where one has 
$\alpha\in{\mathbb Q}$ and $0<x,\alpha x\leq 1$. 
Notice that, by \eqref{K2xxactly} and what was noted just after \eqref{Stirling_Alt}, 
one has 
\begin{equation*} 
{\textstyle\frac12} x \left( K^2 (1,x) - K_2 (x,x)\right) 
= \lim_{\varepsilon\rightarrow 0+} \int_{\varepsilon}^1 \frac{K(x,y) dy}{y} \;,
\end{equation*} 
for $0<x\leq 1$. Having obtained this result via a somewhat indirect route, we 
are curious to know if there exists a more direct proof of it. 
\end{remarks} 
\begin{definitions} 
For $r\in (-1,\infty)$ and $a,b\in [0,1]$, we put  
\begin{equation*}
\Delta_r (a,b) := 
\int_0^1 \left| K(a,z) - K(b,z)\right| z^r dz 
\end{equation*} 
(the existence of this integral following from \eqref{HalfHilbertSchmidt}, for $p=1$, 
combined with the fact that $K$ is bounded on $[0,1]\times [0,1]$). 
\end{definitions}
Clearly $\Delta_r (b,b)=0$ for $r\in (-1,\infty)$ and $0\leq b\leq 1$. 
We have also the following lemma. 
\begin{lemma} 
Let $0<a_0\leq 1$. Suppose, moreover, that $0 < a_n \leq 1$ for all $n\in {\mathbb N}$, and 
that one has $\lim_{n\rightarrow\infty} a_n = a_0$. 
Then, for all $r\in (-1,\infty)$, 
one has $\lim_{n\rightarrow\infty} \Delta_r \left( a_n , a_0\right) = 0$. 
\end{lemma} 
\begin{proof} 
Suppose that $r>-1$. 
It follows from Definitions~2.5 and Equation~\eqref{DefK}, via a couple of changes of 
the variable of integration, that, for $n\in{\mathbb N}$, one has  
\begin{align}\label{Beweis-a1}
\Delta_r\left( a_n , a_0\right) &= \int_0^1 \left| K\left( a_n z , 1\right) - K\left( a_0 z , 1\right) \right| z^r dz 
\nonumber\\ 
 &= \int_0^{\infty} \left| K\left( e^{-(u+A_n)} , 1\right) - K\left( e^{-(u+A_0)} , 1\right) \right|  
e^{-(r+1)u} du\nonumber\\ 
 &= \frac{1}{a_0^{r+1}} \int_{A_0}^{\infty} \left| e^{(r+1)\delta_n} f_r\left( t + \delta_n\right) - f_r(t)\right| dt , 
\end{align} 
where $A_m = \log\left( 1/a_m\right)\in [0,\infty)$ ($m=0,1,2,\ldots\ $), 
$\delta_n = A_n - A_0\in{\mathbb R}$ ($n=1,2,3,\ldots\ $) 
and $f_r$ is the function defined on ${\mathbb R}$ by:
\begin{equation}\label{Beweis-a2} 
f_r(t) := \begin{cases} 
e^{-(r+1)t} K\left( e^{-t} , 1\right) & \text{if $t\geq 0$} , \\ 0 & \text{otherwise} . 
\end{cases} 
\end{equation} 
\par
By another change of variable, it follows from \eqref{Beweis-a2} and 
\eqref{DefK} that one has 
\begin{equation*} 
\int_{-\infty}^{\infty} \left| f_r(t)\right| dt = \int_0^1 \left| K(x,1)\right| x^r dx < \infty  
\end{equation*} 
(given that $r>-1$), so that $f_r$ is Lebesgue integrable on ${\mathbb R}$ (i.e. $f_r\in L^1 ({\mathbb R})$). 
\par
We observe now that, by the triangle inequality, it follows from 
\eqref{Beweis-a1} that one has 
\begin{equation}\label{Beweis-a3} 
0 \leq \Delta_0 \left( a_n , a_0\right) \leq \frac{c_n (r) + d_n (r)}{a_0^{r+1}} 
\qquad\text{($n\in{\mathbb N}$),} 
\end{equation} 
where: 
\begin{equation}\label{Beweis-a4} 
0\leq c_n (r) = \left| e^{(r+1)\delta_n} - 1 \right| \cdot \int_{A_0}^{\infty} 
\left| f_r\left( t + \delta_n\right)\right| dt 
\leq \left| e^{(r+1)\delta_n} - 1 \right| \cdot \int_{-\infty}^{\infty} \left| f_r (t)\right| dt  
\end{equation} 
and 
\begin{equation}\label{Beweis-a5} 
0\leq d_n (r) = \int_{A_0}^{\infty} \left| f_r\left( t + \delta_n\right) - f_r(t)\right| dt 
\leq \int_{-\infty}^{\infty} \left| f_r\left( t + \delta_n\right) - f_r(t)\right| dt . 
\end{equation} 
Since $\lim_{n\rightarrow\infty} a_n = a_0 > 0$, we have here 
$\lim_{n\rightarrow\infty} A_n = \log\left( 1 / a_0\right) = A_0$,  
so that $\lim_{n\rightarrow\infty} \delta_n = A_0 - A_0 = 0$. 
Therefore, given that $f_r$ is independent of $n$, and satisfies $f_r\in L^1 ({\mathbb R})$, 
it follows by \eqref{Beweis-a5} and the case $p=1$ of 
\cite[Theorem~8.19]{Wheeden and Zygmund 2015} that we have $\lim_{n\rightarrow\infty} d_n (r) = 0$. 
Moreover, since $0=\exp\left( \lim_{n\rightarrow\infty} (r+1)\delta_n\right) - 1  
= \lim_{n\rightarrow\infty} \left( \exp\left( (r+1)\delta_n\right) - 1\right)$, 
it follows from \eqref{Beweis-a4} that we have $\lim_{n\rightarrow\infty} c_n (r) = 0$. 
By \eqref{Beweis-a3} and our last two findings, we can deduce (as was required) that 
$\Delta_r \left( a_n , a_0\right) \rightarrow 0$ as $n\rightarrow\infty$. 
\end{proof} 
\begin{theorem} 
The function $K_2$ is continuous on 
$\left( [0,1]\times [0,1]\right) \backslash \left\{ (0,0)\right\}$. 
\end{theorem} 
\begin{proof} 
Since the kernel function $K_2 (x,y)$ is symmetric, it will be 
enough to show that it is continuous on $[0,1]\times (0,1]$. 
Note, moreover, that the definitions \eqref{DefK} and \eqref{K2symmetrically} imply that 
$K_2(0,y)$ is constant for $0\leq y\leq 1$, and so 
we need only show that one has $K_2(x_n,y_n)\rightarrow K_2(x,y)$ as $n\rightarrow\infty$, 
whenever it is the case that $(x_1,y_1),(x_2,y_2),(x_3,y_3),\ldots $ 
is a sequence of points in $(0,1]\times (0,1]$ 
that converges (with respect to the Euclidean metric) to a limit $(x,y)\in [0,1]\times (0,1]$. 
In the case just described one necessarily has both $x_n\rightarrow x\in [0,1]$ 
and $y_n\rightarrow y\in (0,1]$, 
in the limit as $n\rightarrow\infty$. 
By \eqref{K2symmetrically}, one has, moreover, 
\begin{align}\label{K2Sandwich} 
\left| K_2 \left( x_n,y_n \right) - K_2 (x,y)\right| &\leq  
\left| K_2 \left( x_n,y_n \right) - K_2 \left( x ,y_n \right)\right|  + 
\left| K_2 \left( x , y_n \right) - K_2 (x,y) \right| \nonumber\\
 &\leq  \int_0^1 \left| K\left( x_n , z \right) -  K(x,z)\right|\cdot \left| K\left( y_n , z\right)\right| dz \nonumber\\ 
 &\phantom{{\leq}} + 
\int_0^1 \left| K\left( x , z\right)\right| \cdot \left| K\left( y_n , z \right) -  K(y,z)\right| dz \nonumber\\ 
 &\leq {\textstyle\frac{1}{2}} \Delta_0 \left( x_n , x\right) + {\textstyle\frac{1}{2}} \Delta_0 \left( y_n , y\right)
\end{align}
(the last inequality following since, by \eqref{DefK}, 
$K$ has range $( -\frac{1}{2} , \frac{1}{2} ]$). 
By application of the case $r=0$ of Lemma~2.6, we find that when $x,y\in (0,1]$ one has both 
$\Delta_0 \left( x_n , x\right)\rightarrow 0$ and $\Delta_0 \left( y_n , y\right)\rightarrow 0$, 
as $n\rightarrow\infty$. By this and \eqref{K2Sandwich}, 
it follows that, when $x,y\in (0,1]$, one does have 
$\lim_{n\rightarrow\infty} K_2 \left( x_n,y_n \right) = K_2 (x,y)$ (as required).  
\par
In the remaining cases, where $x=0$ and $y\in (0,1]$, we note 
that we have $K_2 (x,y) = K_2 (0,y) = 0$, so that this proof will be complete  
once we are able to show that $K_2\left( x_n , y_n\right) \rightarrow 0$ as $n\rightarrow\infty$. 
With this in mind, we observe (firstly) that 
we have here $\lim_{n\rightarrow\infty} x_n / y_n = x / y = 0 / y = 0$, 
and (secondly) that the result \eqref{K2BoundOffDiagonal} of Lemma~2.4 implies that   
$\left| K_2 \left( x_n , y_n\right)\right| < x_n / y_n$ for all $n\in{\mathbb N}$. 
This shows that 
$\lim_{n\rightarrow\infty} K_2 \left( x_n , y_n\right) = 0$, 
so the proof is complete.  
\end{proof} 
\begin{corollary} 
For any constant $a\in[0,1]$, the functions $y\mapsto K_2 (a,y)$ and $y\mapsto K_2 (y,a)$ 
are continuous on $[0,1]$. 
\end{corollary} 
\begin{proof} 
The cases with $0<a\leq 1$ follow immediately from Theorem~2.7:  
as for the remaining case, where one has $a=0$, it is enough that we observe   
that one has $K_2 (0,y) = K_2 (y,0)=0$ for $0\leq y\leq 1$. 
\end{proof} 
\begin{theorem} 
The function $K_2 : [0,1]\times[0,1]\rightarrow {\mathbb R}$ is not continuous at the point $(0,0)$.  
Furthermore, the set of functions $f : [0,1]\times[0,1]\rightarrow {\mathbb R}$ 
such that the set ${\cal I}_f := \{ (x,y) \in [0,1]\times [0,1] : f(x,y) = K_2 (x,y)\}$ is dense in $[0,1]\times [0,1]$ does not contain one that is continuous at the point $(0,0)$. 
\end{theorem} 
\begin{proof} 
For $f=K_2$, one has ${\cal I}_f = [0,1]\times [0,1]$. The first part of the theorem  
is therefore implied by the second part, and so a proof of the second part is 
all that is required. 
\par 
We adopt the method of `proof by contradiction'. Suppose that the 
second part of the theorem is false. There must then exist a function 
$f : [0,1]\times[0,1]\rightarrow {\mathbb R}$ that is continuous at $(0,0)$ and that (at the same time) 
satisfies $f(x,y)=K_2(x,y)$ for a set of points $(x,y)$ that is dense in $[0,1]\times[0,1]$. 
It follows from the latter part of this that if $\alpha\in (0,1]$ and $(\varepsilon_n)$ is an infinite 
sequence of positive numbers, then  there exists, for each $n\in{\mathbb N}$, 
some pair of real numbers $x_n,y_n$ satisfying both 
\begin{equation}\label{near1overN} 
\frac{1}{n+\varepsilon_n} < x_n < \frac{1}{n} 
\quad \text{and}\quad 
\frac{\alpha}{n+\varepsilon_n} < y_n < \frac{\alpha}{n} 
\end{equation} 
and 
\begin{equation}\label{fEqualK2}
f\left( x_n , y_n\right) = K_2\left( x_n , y_n\right) . 
\end{equation} 
\par 
Let $\alpha$ satisfy $0<\alpha\leq 1$. 
By Theorem~2.7 the function $K_2$ is continuous at each point in 
the sequence 
$(1 , \alpha ),(\frac{1}{2},\frac{\alpha}{2}),(\frac{1}{3},\frac{\alpha}{3}),\ldots $ . 
Therefore, for each $n\in{\mathbb N}$, there exists some number $\varepsilon_n > 0$ such 
that one has 
\begin{equation}\label{K2near1overN}
(x,y)\in \left( \frac{1}{n+\varepsilon_n} , \frac{1}{n} \right) \times 
\left(\frac{\alpha}{n+\varepsilon_n} , \frac{\alpha}{n} \right) 
\Longrightarrow  
\left| K_2 (x,y) - K_2\left( \frac{1}{n} , \frac{\alpha}{n}\right)\right| < \frac{1}{n} . 
\end{equation} 
Thus (bearing in mind the conclusions of the previous paragraph) we deduce the existence of 
sequences, $\varepsilon_1,\varepsilon_2,\varepsilon_3,\ldots $ and 
$(x_1,y_1),(x_2,y_2),(x_3,y_3),\ldots $, such that when $n\in{\mathbb N}$ one has 
both $\varepsilon_n > 0$ and what is stated in \eqref{near1overN}, \eqref{fEqualK2} and  
\eqref{K2near1overN}. Considering now any one such choice of this pair of sequences, 
it follows by \eqref{near1overN}-\eqref{K2near1overN} that, for all $n\in{\mathbb N}$, one has  
\begin{equation*} 
\left| f \left( x_n , y_n \right) - K_2\left( \frac{1}{n} , \frac{\alpha}{n}\right)\right| = 
\left| K_2 \left( x_n , y_n \right) - K_2\left( \frac{1}{n} , \frac{\alpha}{n}\right)\right| < \frac{1}{n} .  
\end{equation*}
Since $f$ is continuous at $(0,0)$, and since \eqref{near1overN} implies 
that $(x_n,y_n)\rightarrow (0,0)$ (with respect to the Euclidean metric) as $n\rightarrow\infty$, 
we have here $\lim_{n\rightarrow\infty} f(x_n,y_n) = f(0,0)$, and so (given that 
$\lim_{n\rightarrow\infty} \frac{1}{n} = 0$) are able to conclude that 
\begin{equation*} 
f(0,0) = \lim_{n\rightarrow\infty} K_2\left( \frac{1}{n} , \frac{\alpha}{n}\right) . 
\end{equation*} 
Note that $\alpha$ here denotes an arbitrary point in the interval $(0,1]$, 
so that it has now been established that the last equality above 
holds for all $\alpha\in (0,1]$. 
By considering the special case $\alpha = 1$, 
we deduce that 
\begin{equation*} 
f(0,0) = \lim_{n\rightarrow\infty} K_2\left( \frac{1}{n} , \frac{1}{n}\right) 
= \lim_{n\rightarrow\infty} {\textstyle\frac{1}{12}} \cdot \left( 1 + O\left( n^{-1}\right)\right)  
= {\textstyle\frac{1}{12}}  
\end{equation*} 
(the middle equality here holding by virtue of \eqref{K2BoundOnDiagonal}). 
Therefore, for each fixed choice of $\alpha\in (0,1]$, 
we have $\lim_{n\rightarrow\infty} K_2\left( n^{-1} , \alpha n^{-1}\right) = \frac{1}{12}$. 
The result \eqref{K2BoundOffDiagonal}, however, shows that one has 
$\left| K_2\left( \frac{1}{n} , \frac{\alpha}{n}\right)\right| 
\leq \bigl( \frac{1}{4} + \frac{1}{36 \sqrt{3}}\bigr) \alpha < \frac{4}{15}\alpha$ 
for $0<\alpha\leq 1$, $n\in{\mathbb N}$. In particular, when $\alpha = \frac{5}{16}$ (for example), 
one has 
$\frac{1}{12}=\frac{4}{15}\alpha > \sup\bigl\{ K_2\left( \frac{1}{n} , \frac{\alpha}{n}\right) : n\in{\mathbb N}\bigr\} $. 
This is incompatible with our earlier finding that 
$\lim_{n\rightarrow\infty} K_2\left( n^{-1} , \alpha n^{-1}\right) = \frac{1}{12}$ 
if $0<\alpha\leq 1$. In light of the contradiction evident here, 
we are left with no option but to conclude that 
the second part of the theorem cannot be false; we have therefore shown it to be (instead) true, which 
is all that we need to complete this proof.
\end{proof} 
\begin{theorem} All eigenfunctions of $K$ (including, in particular, the functions 
$\phi_1,\phi_2,\phi_3,\ldots $) are continuous on $[0,1]$. 
\end{theorem} 
\begin{proof} 
It will be enough to show that one has 
\begin{equation}\label{Beweis-b1} 
\lim_{n\rightarrow\infty} \phi\left( x_n\right) = \phi\left( x_0\right) 
\end{equation}
if $\phi$ is an eigenfunction of $K$ and   
$x_0,x_1,x_2,\ldots\ $ a sequence of elements of $[0,1]$ satisfying $x_n\rightarrow x_0$ 
as $n\rightarrow\infty$. Accordingly, we suppose now that the conditions just mentioned 
(after \eqref{Beweis-b1}) are satisfied. 
By \eqref{K2Eigenfunction}, we have 
\begin{equation}\label{Beweis-b2} 
\phi \left( x_n \right) = \int_0^1 f_n (y) dy\qquad\text{($n=0,1,2,\ldots\ $),} 
\end{equation} 
where $f_n (y) := \lambda^2 K_2\left( x_n , y\right) \phi (y)$, with $\lambda$ 
being the relevant eigenvalue of $K$. 
\par
From \eqref{Beweis-b2} it follows (implicitly) that all functions in 
the sequence $f_1,f_2,f_3,\ldots\ $ are measurable on $[0,1]$. 
Note also that, by Corollary~2.8, we have 
$K_2\left( x_0 , y\right) = K_2\left( \lim_{n\rightarrow\infty} x_n , y\right) 
= \lim_{n\rightarrow\infty} K_2\left( x_n , y\right)$ for $0\leq y\leq 1$, and 
so it is certainly the case that one has $\lim_{n\rightarrow\infty} f_n (y) = f_0 (y)$ 
almost everywhere in $[0,1]$. In view of the two points just noted, it 
follows by Lebesgue's `Dominated Convergence Theorem' \cite[Theorem~5.36]{Wheeden and Zygmund 2015} 
that, if there exists a function $F$ that is integrable on $[0,1]$ and 
satisfies $F(y)\geq \sup\left\{ \left| f_n (y)\right| : n\in{\mathbb N}\right\}$ 
almost everywhere in $[0,1]$, then one will have 
\begin{equation*} 
\int_0^1 f_0 (y) dy = \int_0^1 \left( \lim_{n\rightarrow\infty} f_n (y)\right) dy 
=\lim_{n\rightarrow\infty} \int_0^1 f_n (y) dy . 
\end{equation*} 
This last outcome would immediately imply, by virtue of \eqref{Beweis-b2}, that 
the equality in \eqref{Beweis-b1} does indeed hold. Therefore, in order to complete 
this proof, we have only to observe now that the function 
$F(y) := \frac{1}{4} \lambda^2 |\phi (y)|$ is integrable over $[0,1]$ 
(the fact that we have $\phi\in L^2\bigl( [0,1]\bigr)$ implies this, since the interval $[0,1]$ is bounded), 
and that, from the definition of $f_n$ and the bound $\left| K_2 (x,y)\right| < \frac{1}{4}$ 
($0\leq x,y\leq 1$), implied by \eqref{DefK} and \eqref{K2symmetrically}, it follows  that the same function $F$ 
satisfies $F(y)\geq \sup\left\{ \left| f_n (y)\right| : n\in{\mathbb N}\right\}$ for $0\leq y\leq 1$.
\end{proof} 
\begin{remarks}{\it 1)} 
Let $j\in{\mathbb N}$. Then, by \eqref{DefEigenfunction} and \eqref{DefK},  
one has $\phi_j (0) = 0$. Thus it follows from Theorem~2.10 that 
one has $\lim_{x\rightarrow 0+} \phi_j (x) = 0$. In the next section we 
discover more about how the eigenfunctions $\phi_1 (x),\phi_2 (x),\phi_3 (x), \ldots\ $ 
behave as $x$ tends towards $0$ from above. 
\item{\it 2)} We need Theorem~2.7, Corollary~2.8 and Theorem~2.10 for the proof of our next result, 
the `bilinear formula' for $K_2$. In \cite[Sections~3.9, 3.10 and~3.12]{Tricomi 1957} and 
\cite[Sections~7.3 and~7.4]{Kanwal 1997} (for example), it is shown that 
the bilinear formula for a kernel $k(x,y)$ is valid if the function $k$ satisfies certain conditions. 
Yet, neither of these two references, nor any other that we know of, quite 
manages to cover the case of our kernel $K_2$: the discontinuity of $K_2 (x,y)$ at the point 
$(x,y)=(0,0)$ prevents this.  
\end{remarks} 
\begin{theorem} Let $0<\varepsilon\leq 1$. 
Then the series 
\begin{equation}\label{BilinearSeries} 
\frac{\phi_1 (x) \phi_1 (y)}{\lambda_1^2} +  
\frac{\phi_2 (x) \phi_2 (y)}{\lambda_2^2} +  
\frac{\phi_3 (x) \phi_3 (y)}{\lambda_3^2} +\ \ldots\ 
\end{equation}  
converges uniformly for $(x,y)\in [0,1]^2 \backslash (0,\varepsilon)^2$. 
For $0\leq x,y\leq 1$, this series is absolutely convergent, and one has:  
\begin{equation}\label{K2BilinearFormula} 
\sum_{j=1}^{\infty} \frac{\phi_j (x) \phi_j(y)}{\lambda_j^2} = 
K_2 (x,y) .   
\end{equation} 
\end{theorem} 
\begin{proof} 
Let $x_1\in [0,1]$. Put $f(y) := K_2 \left( x_1 , y\right)$, 
so that for $0\leq y\leq 1$ one has 
$f(y) = \int_0^1 K(y,z) g(z) dz$, where $g(z):=K(x_1,z)$. 
By Corollary~2.8, the function $f$ is continuous on $[0,1]$. 
Since the kernel $K$ is a measurable function on $[0,1]\times [0,1]$ 
that satisfies both \eqref{HSnorm} and \eqref{HalfHilbertSchmidt}, it follows 
by an application \cite[Page~113]{Tricomi 1957} of the `Hilbert-Schmidt theorem' 
\cite[Page~110]{Tricomi 1957} that the series \eqref{BilinearSeries} 
converges, both absolutely and uniformly, for $(x,y)\in \left\{ x_1\right\}\times [0,1]$, and 
that the corresponding sums, $F(y) := \sum_{j=1}^{\infty} \lambda_j^{-2} \phi_j \left( x_1\right) \phi_j (y)$ 
($0\leq y\leq 1$), satisfy $F(y) = f(y)$ almost everywhere in $[0,1]$, so that 
the set $\left\{ y \in [0,1] : F(y) = f(y)\right\}$ is certainly dense in $[0,1]$. 
\par 
For $x=x_1$, each partial sum of the series \eqref{BilinearSeries} 
is a linear combination of finitely many of the 
eigenfunctions $\phi_1 (y),\phi_2 (y), \ldots\ $, and so, by 
Theorem~2.10, is a function of $y$ that is continuous on $[0,1]$. 
Therefore, given that we know these partial sums to be the terms of a sequence 
converging uniformly to the limit $F(y)$ on $[0,1]$, it follows 
that that limit, $F$, is continuous on $[0,1]$. 
Thus, both $f$ and $F$ are continuous on $[0,1]$, so that the set 
$\left\{ y \in [0,1] : F(y) = f(y)\right\}$, being dense in $[0,1]$, 
must contain the interval $[0,1]$. That is, we have \eqref{K2BilinearFormula} 
for $x=x_1$ and all $y\in [0,1]$. 
\par
Since $x_1$ here denotes an arbitrary fixed point in the interval $[0,1]$, 
it has now been established that, for $0\leq x,y\leq 1$, 
the equality \eqref{K2BilinearFormula} holds and the infinite sum 
occurring in \eqref{K2BilinearFormula} converges absolutely. 
\par
We now have only to prove the part of the theorem 
concerning uniform convergence on $[0,1]^2\backslash (0,\varepsilon)^2$. 
We begin by observing that, since one has $\phi_j (x) \phi_j (y) = \phi_j (y) \phi_j (x)$ 
for $j\in{\mathbb N}$ and $0\leq x,y\leq 1$, it will be enough to establish 
that the series \eqref{BilinearSeries} converges uniformly 
for $x\in \{ 0\}\cup [\varepsilon , 1]$, $y\in [0,1]$. 
We know, from the first paragraph of this proof (for example), that the series 
\eqref{BilinearSeries} does converge uniformly for $x=0$ and $0\leq y\leq 1$. 
All that now remains to be shown is that the series \eqref{BilinearSeries} converges uniformly 
for $(x,y)\in [\varepsilon , 1]\times [0,1]$. 
To this end, we note that, by the Cauchy-Schwarz inequality and  \eqref{K2BilinearFormula}, 
\eqref{K2symmetrically} and \eqref{HalfHilbertSchmidt}, it follows that, when $N\in{\mathbb N}$, one has: 
\begin{align*} 
\left| \sum_{j=N+1}^{\infty} \frac{\phi_j (x) \phi_j(y)}{\lambda_j^2}  \right| &\leq 
\left( \sum_{j=N+1}^{\infty} \frac{\phi_j^2 (x)}{\lambda_j^2} \right)^{\frac{1}{2}} 
\left( \sum_{j=N+1}^{\infty} \frac{\phi_j^2 (y)}{\lambda_j^2} \right)^{\frac{1}{2}} \\ 
 &\leq \left( K_2 (y,y)\right)^{\frac{1}{2}} 
\left( \sum_{j=N+1}^{\infty} \frac{\phi_j^2 (x)}{\lambda_j^2} \right)^{\frac{1}{2}} \leq 
{\textstyle\frac{1}{2}} \left( \sum_{j=N+1}^{\infty} \frac{\phi_j^2 (x)}{\lambda_j^2} \right)^{\frac{1}{2}} ,
\end{align*}  
for $0\leq x,y\leq 1$. 
Therefore, all that we now have to do (in order to complete this proof) is show that the series 
$\lambda_1^{-2} \phi_1^2 (x) + \lambda_2^{-2} \phi_2^2 (x) + \lambda_3^{-2} \phi_3^2 (x) + \ \ldots\ $ 
converges uniformly for $x\in [\varepsilon,1]$.  
\par
Putting $s_n (x) := \sum_{j=1}^N \lambda_j^{-2} \phi_j^2 (x)$ ($N\in{\mathbb N}$, $0\leq x\leq 1$), we
observe that $s_1 (x)\leq s_2 (x) \leq s_3 (x) \leq\ \ldots\ $ ($0\leq x\leq 1$), that 
the functions $s_1,s_2,s_3,\ldots\ $ are continuous on $[0,1]$ (by virtue of Theorem~2.10), 
and that, by \eqref{K2BilinearFormula}, it follows that, for $0\leq x\leq 1$, one has 
$\lim_{N\rightarrow\infty} s_N (x) = K_2 (x,x)$, which, by Theorem~2.7, is a continuous 
function of $x$ on the interval $[\varepsilon , 1]$. 
By Dini's theorem \cite[Theorem~7.13]{Rudin 1976}, it follows from what we have just noted 
that the sequence $s_1 (x), s_2 (x), s_3(x) , \ldots\ $ is uniformly convergent on 
the compact interval $[\varepsilon , 1]\subset [0,1]$: this means, of course, 
that the same is true of the series 
$\lambda_1^{-2} \phi_1^2 (x) + \lambda_2^{-2} \phi_2^2 (x) + \lambda_3^{-2} \phi_3^2 (x) + \ \ldots\ $. 
\end{proof}
\begin{remarks} 
The above proof is, in essence, an adaptation of the proof 
of `Mercer's theorem' that appears in \cite[Section~3.12]{Tricomi 1957}.
\end{remarks} 
\begin{corollary} 
One has 
\begin{equation}\label{K2diagonal} 
\sum_{j=1}^{\infty} 
\frac{\phi_j^2 (x)}{\lambda_j^2} = 
K_2 (x,x) = 
\int_0^1 K^2 (x,z) dz \leq \frac{1}{4} 
\quad\text{($0\leq x\leq 1$),} 
\end{equation} 
and (in consequence of this) one has also 
\begin{equation}\label{PhiBoundedUniformly}
\left| \phi_j (x)\right| \leq 
{\textstyle\frac{1}{2}} \left| \lambda_j\right| 
\quad\text{($0\leq x\leq 1$, $j\in{\mathbb N}$),} 
\end{equation} 
\begin{equation}\label{K(x,y)_in_mean}
\lim_{H\rightarrow\infty} \int_0^1 \left( K(x,y) - 
\sum_{h=1}^H \frac{\phi_h (x) \phi_h(y)}{\lambda_h} \right)^2 dy = 0 
\quad\text{($0\leq x\leq 1$)} 
\end{equation} 
and
\begin{equation}\label{K2Trace}
\sum_{h=1}^{\infty} \frac{1}{\lambda_h^2} = \| K\|_{\rm HS}^2\;. 
\end{equation} 
\end{corollary}
\begin{proof} 
The result \eqref{K2diagonal} follows immediately from Theorem~2.11, 
\eqref{K2symmetrically} and \eqref{HalfHilbertSchmidt}, for $p=2$. 
By \eqref{K2diagonal}, we have $\frac{1}{4}\geq \lambda_j^{-2} \phi_j^2 (x)$, 
for $j\in{\mathbb N}$ and $0\leq x\leq 1$.  
From this, we immediately obtain the bounds \eqref{PhiBoundedUniformly}. 
\par 
Suppose now that $x\in [0,1]$. 
By expanding the integrand in \eqref{K(x,y)_in_mean} and integrating term by term, we 
find (using \eqref{DefEigenfunction}, \eqref{K2symmetrically} 
and the orthonormality of $\phi_1,\phi_2,\phi_3,\ldots $ ) 
that the limit occurring in \eqref{K(x,y)_in_mean}  
is $\lim_{H\rightarrow\infty} \bigl( K_2(x,x) -  \sum_{h=1}^H \lambda_h^{-2} \phi_h^2 (x)\bigr)$,  
which (by \eqref{K2diagonal}) is equal to $0$. This proves \eqref{K(x,y)_in_mean}. 
\par 
By the `Monotone Convergence Theorem' \cite[Theorem~5.32]{Wheeden and Zygmund 2015}, 
it follows from the result \eqref{K2diagonal} that one has  
\begin{equation*}
\lim_{H\rightarrow\infty} \int_0^1 \biggl(\sum_{h=1}^H \frac{\phi_h^2 (x)}{\lambda_h^2}\biggr) dx  
= \int_0^1 K_2 (x,x) dx . 
\end{equation*} 
By this, combined with both the 
fact that $\left\| \phi_h \right\| = 1$ ($h\in{\mathbb N}$)
and the definitions in \eqref{K2symmetrically} and \eqref{HSnorm}, 
we obtain what is stated in \eqref{K2Trace}.  
\end{proof} 
\begin{remarks} 
{\it 1)} By \eqref{K2Trace}, \eqref{HSnorm} and \eqref{EigenvalueOrder1}, we have:  
\begin{equation}\label{SpectrumBound} 
\left|\lambda_1\right| = \min\left\{ \left| \lambda_j\right| : j\in{\mathbb N}\right\} 
> \| K\|_{\rm HS}^{-1} > 2\;. 
\end{equation} 
\item{\it 2)} For an alternative proof of \eqref{PhiBoundedUniformly}, simply bound the integral in \eqref{DefEigenfunction} using the Cauchy-Schwarz inequality, \eqref{HalfHilbertSchmidt} 
and the relation $\| \phi_j\| = 1$. 
\item{\it 3)} Since all the summands occurring in the series 
$\sum_{j=1}^{\infty} \lambda_j^{-2} \phi_j^2 (x)$ are non-negative real numbers, 
it can be deduced from Theorem~2.11 that, for any constant $\varepsilon\in (0,1)$, the sequence $\bigl( \lambda_j^{-1} \phi_j (x)\bigr)$ converges uniformly (to the limit $0$) for all $x\in [\varepsilon , 1]$. 
That is, for each $\varepsilon\in (0,1)$, one has 
\begin{equation*} 
\left| \phi_j (x)\right| \leq 
\frac{\left| \lambda_j\right|}{d_j (\varepsilon)} 
\quad\text{($\varepsilon\leq x\leq 1$ and $j\in{\mathbb N}$),} 
\end{equation*}
where $( d_j (\varepsilon) )$ is some 
unbounded monotonic increasing sequence of positive numbers that depends only on $\varepsilon$ 
(by \eqref{PhiBoundedUniformly}, one can assume that $d_1(\varepsilon)\geq 2$). 
\item{\it 4)} The series 
$\lambda_1^{-2} \phi_1^2 (x) + \lambda_2^{-2} \phi_2^2 (x) + \lambda_3^{-2} \phi_3^2 (x) + \ \ldots\ $ 
is not uniformly convergent on $[0,1]$. 
If it were, then the function $K_2$ would be continuous on 
$[0,1]\times [0,1]$ (this would follow by virtue of Theorem~2.10, Theorem~2.11 and 
the inequalities that are obtained in the penultimate paragraph of the proof of Theorem~2.11). 
By Theorem~2.9, however, we know that $K_2$ is certainly not continuous at 
the point $(0,0)\in [0,1]\times [0,1]$.  
\end{remarks} 

\section{Lipschitz conditions} 

\bigskip 

\begin{lemma} 
When $0<x\leq 1$, one has: 
\begin{equation*} 
\int_x^1 
\Biggl| \sum_{m > \frac{1}{y}} 
\frac{\widetilde B_2 \left( \frac{my}{x}\right)}{m^2} 
\Biggr| 
\frac{dy}{y^2} < \frac{2}{3} . 
\end{equation*} 
\end{lemma} 
\begin{proof} 
Let $0<x<1$, and define $X := x^{-1}$, so that $X>1$. 
Then, by considering the effect of the substitution $y=Y^{-1}$, 
we find that the lemma will follow if it can be shown that one has
\begin{equation}\label{proofA1}
b(X) := 
\int_1^X 
\Biggl| \sum_{m > Y} 
\frac{\widetilde B_2 \left( \frac{mX}{Y}\right)}{m^2} 
\Biggr| dY < \frac{2}{3} . 
\end{equation}
\par 
Supposing now that $n$ is a positive integer satisfying $n<X$, we 
put $\nu_n := \min\{ n+1 , X\}$. By applying the Levi theorem for series 
\cite[Theorem~10.26]{Apostol 1974}, one can establish that the function 
$Y\mapsto \sum_{m > Y} m^{-2} \widetilde B_2 \left( m X Y^{-1}\right)$ is Lebesgue integrable 
on the interval $[n,\nu_n)$. We therefore may define 
\begin{equation}\label{proofA2} 
b_n(X) := 
\int_n^{\nu_n}  
\Biggl| \sum_{m > Y} 
\frac{\widetilde B_2 \left( \frac{mX}{Y}\right)}{m^2} 
\Biggr| dY\;. 
\end{equation} 
\par 
Now, for $Y\in [n,\nu_n)$, it follows by Definitions~2.2 and \cite[Equations~24.8.1]{Olver et al. 2010} 
that one has 
\begin{equation}\label{proofA3} 
\Biggl| \sum_{m > Y} 
\frac{\widetilde B_2 \left( \frac{mX}{Y}\right)}{m^2} 
\Biggr| = 
\Biggl| \sum_{m > Y}   
\sum_{h=1}^{\infty} \frac{\cos\left( 2\pi h m X Y^{-1}\right)}{\pi^2 h^2 m^2} \Biggr|  
\leq \sum_{h=1}^{\infty} g_h (Y) ,
\end{equation} 
where 
\begin{equation*} 
g_h (Y) := \Biggl| 
 \sum_{m > Y} \frac{\cos\left( 2\pi h m X Y^{-1}\right)}{ \pi^2 h^2 m^2} \Biggr| = 
\Biggl| \sum_{m=n+1}^{\infty}\frac{\cos\left( 2\pi h m X Y^{-1}\right)}{\pi^2 h^2 m^2} \Biggr|  
\end{equation*} 
(the inequality in \eqref{proofA3} being justified by the fact that the double series 
occurring there is absolutely convergent --- so that one may, in particular, 
change the original order of summation by summing firstly over $m$).
By \cite[Theorems~10.26 and~10.16]{Apostol 1974}, each member of the sequence $g_1(Y),g_2(Y),g_3(Y),\ldots $ 
is a function that is Lebesgue integrable on the interval $[n,\nu_n)$. 
We have, moreover, 
\begin{equation*} 
g_h(Y)\geq 0\quad\text{($h\in{\mathbb N}$ and $n\leq Y<\nu_n$)} 
\end{equation*} 
and 
\begin{equation*} 
\sum_{h=1}^{\infty} g_h(Y) \leq 
\sum_{h=1}^{\infty} \sum_{m=n+1}^{\infty} \frac{1}{\pi^2 h^2 m^2} 
<\left( \sum_{k=1}^{\infty} \frac{1}{\pi k^2}\right)^2 < \infty ,  
\end{equation*} 
and so (bearing in mind also that $[n,\nu_n)$ is a bounded interval) we are 
able to conclude that it follows by Lebesgue's dominated convergence theorem 
\cite[Theorem~10.28]{Apostol 1974} that the function 
$Y\mapsto \sum_{h=1}^{\infty} g_h (Y)$ is Lebesgue integrable on $[n,\nu_n)$, 
and that one has 
\begin{equation*} 
\int_n^{\nu_n} \left( \sum_{h=1}^{\infty} g_h (Y)\right) dY = 
\sum_{h=1}^{\infty} \int_n^{\nu_n} g_h (Y) dY .
\end{equation*} 
By this, combined with \eqref{proofA2} and \eqref{proofA3}, it follows that one has 
\begin{equation*} 
b_n(X) \leq \sum_{h=1}^{\infty} \int_n^{\nu_n} \Biggl| 
 \sum_{m > Y} \frac{\cos\left( 2\pi h m X Y^{-1}\right)}{ \pi^2 h^2 m^2} \Biggr| dY . 
\end{equation*} 
By summing each side of this last inequality over the finitely many 
choices of the integer $n$ that satisfy the condition $n<X$ we find 
(upon recalling the definitions \eqref{proofA1} and \eqref{proofA2}) that 
\begin{align}\label{proofA4} 
b(X) = \sum_{1\leq n<X} b_n(X) &\leq 
\sum_{h=1}^{\infty} \sum_{1\leq n<X}\int_n^{\nu_n} \Biggl| 
 \sum_{m > Y} \frac{\cos\left( 2\pi h m X Y^{-1}\right)}{ \pi^2 h^2 m^2} \Biggr| dY \nonumber\\ 
 &=\sum_{h=1}^{\infty} \int_1^X \Biggl| 
 \sum_{m > Y} \frac{\cos\left( 2\pi h m X Y^{-1}\right)}{\pi^2 h^2 m^2} \Biggr| dY \nonumber\\ 
 &= 
\sum_{h=1}^{\infty} \frac{J(h,X)}{\pi^2 h^2} , 
\end{align}
where, for $h\in{\mathbb N}$, we have: 
\begin{equation*}
J(h,X) := \int_1^X \Biggl| 
 \sum_{m > Y} \frac{\cos\left( 2\pi m h X Y^{-1}\right)}{m^2} \Biggr| dY . 
\end{equation*} 
\par 
Our next objective is an upper bound for the integrand just seen in our 
definition of $J(h,X)$. Our proof of this bound utilises a method 
well-known to analytic number theorists. 
Let $Y>1$ and $t\in{\mathbb R}\backslash{\mathbb Z}$. 
We have 
\begin{equation}\label{proofA5} 
\sum_{m>Y} \frac{\cos(2 \pi m t)}{m^2} 
=\int_Y^{\infty} u^{-2} dC(u) ,
\end{equation} 
where, for $u\geq 1$, one has 
\begin{equation*} 
C(u) = \sum_{0<m\leq u} \cos(2 \pi m t) 
= {\rm Re}\Biggl( \sum_{m=1}^{\lfloor u\rfloor} e^{2 \pi i m t}\Biggr) 
= {\rm Re}\left( \frac{e^{2 \pi i \lfloor u\rfloor t} - 1}{1 - e^{- 2 \pi i t}}\right) , 
\end{equation*} 
and so 
\begin{equation*} 
\left| C(u)\right| \leq \frac{2}{\left| e^{\pi i t} - e^{-\pi i t}\right|} 
=\frac{1}{\left| \sin(\pi t)\right|} = \frac{1}{\sin\left( \pi \| t\|\right)} 
\leq \frac{1}{2 \| t\|} , 
\end{equation*}
where $\| t\| =\min\{ |t-j| : j\in{\mathbb Z}\}$. 
Using integration by parts, we deduce from \eqref{proofA5} and the above upper 
bound for $|C(u)|$ that one has 
\begin{align*}
\left| \sum_{m>Y} \frac{\cos(2 \pi m t)}{m^2} \right| 
 &=\left| 2\int_Y^{\infty} u^{-3} C(u) du - Y^{-2} C(Y)\right| \\ 
 &\leq \frac{1}{\| t\|} \left( \int_Y^{\infty} u^{-3} du + {\textstyle\frac{1}{2}} Y^{-2}\right) 
= \frac{1}{\| t\| Y^2} .
\end{align*} 
It is also (trivially) the case that   
\begin{equation*} 
\left| \sum_{m>Y} \frac{\cos(2 \pi m t)}{m^2} \right|  
\leq \sum_{m>Y} \frac{1}{m^2} < \frac{1}{Y^2} + \int_Y^{\infty}  \frac{du}{u^2} < \frac{2}{Y} .
\end{equation*} 
The latter bound remains valid for integer values of $t$, 
and so (by combining the two bounds just noted) we find that 
one has: 
\begin{equation*} 
\left| \sum_{m>Y} \frac{\cos(2 \pi m t)}{m^2} \right| 
\leq \frac{2}{Y \max\left\{ 1 , 2 Y \| t\|\right\}} 
\quad\text{($Y>1$, $t\in{\mathbb R}$).} 
\end{equation*}
\par 
Given our definition of $J(h,X)$, it follows by the upper 
bounds just obtained that, for $h\in{\mathbb N}$, one has: 
\begin{align}\label{proofA6}  
J(h,X) &\leq 
2 \int_1^X  \left( \max\left\{ 1 , 2 Y \left\| \frac{h X}{Y}\right\|\right\}\right)^{-1} \frac{dY}{Y}\nonumber\\ 
 &= - 2 \int_{hX}^h  \left( \max\left\{ 1 , 2 h X t^{-1} \| t \|\right\}\right)^{-1} \frac{dt}{t}\nonumber\\ 
 &= \frac{1}{hX} \int_h^{hX} \frac{dt}{\max\left\{ \frac{t}{2 h X} , \| t \|\right\}} .
\end{align} 
For each positive integer $k < hX$, we have: 
\begin{align*} 
\int_k^{k+1} \frac{dt}{\max\left\{ \frac{t}{2 h X} , \| t \|\right\}}  
 &\leq \int_{-\frac{1}{2}}^{\frac{1}{2}} \frac{dt}{\max\left\{ \frac{k}{2 h X} , \| t \|\right\}}  \\ 
 &= 2\int_0^{\frac{k}{2hX}} \left( \frac{2hX}{k}\right) dt 
 + 2\int_{\frac{k}{2hX}}^{\frac{1}{2}} \frac{dt}{t} \\ 
 &= 2 + 2\log\left( \frac{hX}{k}\right) .
\end{align*} 
It follows by this and \eqref{proofA6} that, for each $h\in{\mathbb N}$, one has   
\begin{align*} 
J(h,X) &\leq \frac{2}{hX} \sum_{1\leq k < hX} \left( 1 + \log\left( \frac{hX}{k}\right)\right) \\ 
 &< \frac{2}{hX} \left( hX + \log\left( \frac{hX}{1} \right) + \int_1^{hX} \log\left( \frac{hX}{\kappa}\right) d\kappa \right) \\ 
 &=\frac{2}{hX}\left( hX + hX\log(hX) + (-1) - \left( hX\log(hX) - hX\right)\right) \\
 &=\frac{2}{hX}\left( 2hX - 1\right) ,
\end{align*} 
so that $J(h,X) < 4$. By this, \eqref{proofA4} and  
Euler's famous evaluation of the sum $\sum_{h=1}^{\infty} h^{-2}$, we obtain the inequality in \eqref{proofA1}.    
This completes our proof in respect of cases where $0<x<1$. The remaining case 
($x=1$) is trivial.  
\end{proof} 
\begin{theorem} 
For all $j\in{\mathbb N}$, one has  
\begin{equation*} 
\frac{\left| \phi_j (x)\right|}{x} \leq C_0 \left| \lambda_j \right|^3 
\quad\text{($0 < x\leq 1$),} 
\end{equation*} 
where 
\begin{equation}\label{DefC0} 
C_0 := \frac{1}{3} + \frac{1}{72 \sqrt{3} e} . 
\end{equation}
\end{theorem} 
\begin{proof} 
Let $j\in{\mathbb N}$ and $0<x\leq 1$. 
By \eqref{K2Eigenfunction} and \eqref{PhiBoundedUniformly}, we obtain the bound   
\begin{equation}\label{proofB1} 
\frac{\left| \phi_j (x)\right|}{x} \leq  
{\textstyle\frac{1}{2}} \left| \lambda_j\right|^3 
\int_0^1 \frac{\left| K_2 (x,z)\right|}{x} dz .
\end{equation} 
By \eqref{K2OverxBound}, the trivial bound $|K_2 (x,z)| < \frac{1}{4} $ ($0\leq z\leq 1$) and 
Lemma~3.1, we find that 
\begin{align*} 
\int_0^1 \frac{\left| K_2 (x,z)\right|}{x} dz  
 &< \frac{1}{4x} \int_0^x dz  \\ 
 &\phantom{{<}} + \int_x^1 \Biggl( 
 \frac{1}{12} + 
\frac{x}{\left( 36 \sqrt{3}\right) z} + 
\frac{1}{2 z^2} \Biggl| \sum_{m > \frac{1}{z}} 
\frac{\widetilde B_2 \left( \frac{mz}{x}\right)}{m^2} 
\Biggr| 
\Biggr) dz \\ 
 &< \frac{1}{4} + \frac{1}{12} + \frac{x\log\left( x^{-1}\right)}{36\sqrt{3}} + \frac{1}{3} .
\end{align*} 
Here $x\log\left( x^{-1}\right)\leq e^{-1}$ (given that $0<x\leq 1$), and 
so the theorem follows directly from the last bound above and \eqref{proofB1}. 
\end{proof} 

\medskip 

\begin{remarks} 
Let $j\in{\mathbb N}$. In view of our having $\phi_j (0)=0\,$ 
(by \eqref{DefEigenfunction} and \eqref{DefK}), Theorem~3.2 
shows that the eigenfunction $\phi_j(x)$ satisfies a 
right-handed Lipschitz condition of order $1$ at the point $x=0$. 
That is, one has  
\begin{equation*} 
\left| \phi_j (x) - \phi_j (0)\right| \leq M^*_j x\qquad\text{($0\leq x\leq 1$)} , 
\end{equation*} 
with $M^*_j := C_0 |\lambda_j|^3$ independent of $x$. 
\end{remarks} 

\medskip 

\begin{lemma} 
When $0<a,b\leq 1$, one has: 
\begin{equation}\label{Delta1Bound}
0\leq \Delta_1(a,b)  := 
\int_0^1 \left| K(a,z) - K(b,z)\right| z dz \leq 
4\left| \frac{1}{b} - \frac{1}{a}\right|\;.  
\end{equation} 
\end{lemma} 
\begin{proof} 
Since the upper bound in 
\eqref{Delta1Bound} is invariant under the permutation $(a,b)\mapsto (b,a)$, 
we may suppose that $0<b\leq a\leq 1$. 
Given \eqref{DefK} and Definitions~2.5, it is trivially the case that we have 
$0\leq \Delta_1 (a,b) \leq \int_0^1 z dz = \frac{1}{2}$. 
Thus the bound 
\eqref{Delta1Bound} certainly holds if $\frac{1}{b} - \frac{1}{a}\geq \frac{1}{2}$. 
We may therefore assume henceforth that 
\begin{equation}\label{deltaSmall} 
0 \leq \delta := \frac{1}{b} - \frac{1}{a} < \frac{1}{2} . 
\end{equation} 
\par 
By Definitions~2.5, \eqref{DefK} and \eqref{deltaSmall}, 
we find (using the triangle inequality) that 
\begin{align*} 
\Delta_1 (a,b) &= \int_0^1 \left| \left\{ \frac{1}{bz}\right\} 
-  \left\{ \frac{1}{az}\right\} \right| z dz \\ 
 &= \int_0^1 \left| 
 \left( \frac{1}{bz} - \frac{1}{az}\right) + 
\left(\left\lfloor \frac{1}{az}\right\rfloor 
-  \left\lfloor \frac{1}{bz}\right\rfloor\right) \right| z dz \\ 
 &\leq 
\int_0^1 \left( \left(\frac{1}{bz} - \frac{1}{az} \right) + 
\sum_{\frac{1}{az} < n \leq \frac{1}{bz}} 1\right) z dz \\ 
 &= \delta + \sum_{n > \frac{1}{a}}\int_{\frac{1}{an}}^{\min\{ \frac{1}{bn} , 1\}} z dz \\ 
 &= \delta + 
{\textstyle\frac{1}{2}} \sum_{\frac{1}{a} < n \leq \frac{1}{b}} \left( 1 - \frac{1}{a^2 n^2} \right) 
+ {\textstyle\frac{1}{2}}\sum_{n > \frac{1}{b}} \left( \frac{1}{b^2 n^2} - \frac{1}{a^2 n^2} \right)  \\ 
 &\leq \delta + {\textstyle\frac{1}{2}}  \left( 1 - \frac{b^2}{a^2} \right) \sum_{\frac{1}{a} < n \leq \frac{1}{b}} 1 +  
{\textstyle\frac{1}{2}} \left( \frac{1}{b^2} - \frac{1}{a^2}\right) \sum_{n > \frac{1}{b}} \frac{1}{n^2} . 
\end{align*}
This, together with \eqref{deltaSmall} (and our assumption that $0<b<a\leq 1$), yields: 
\begin{align*} 
\Delta_1 (a,b) &\leq \delta + {\textstyle\frac{1}{2}}  
\left( 1 + \frac{b}{a} \right)\left( 1 - \frac{b}{a} \right) \cdot (1) +  
{\textstyle\frac{1}{2}} \left( \frac{1}{b} + \frac{1}{a}\right) \delta \cdot \left( b^2 + b\right) \\ 
& \leq \delta + {\textstyle\frac{1}{2}}  \left( 2 \right) b \delta +  
{\textstyle\frac{1}{2}} \left( \frac{2}{b} \right) \delta \cdot \left( 2 b\right) 
\leq 4\delta ,
\end{align*}
which is \eqref{Delta1Bound}. 
\end{proof}
\begin{lemma} 
Let $C_0$ be the constant defined in \eqref{DefC0}. Let $j\in{\mathbb N}$. 
Then 
\begin{equation*} 
\left| \phi_j(x) - \phi_j(y)\right| \leq 
4 C_0 \lambda_j^4 \cdot \left| \frac{1}{x} - \frac{1}{y}\right| 
\quad\text{($0<x,y\leq 1$).}  
\end{equation*} 
\end{lemma} 
\begin{proof} 
Let $0<x,y\leq 1$. By \eqref{DefEigenfunction} and Theorem~3.2, we find that 
\begin{align*} 
\left| \phi_j(x) - \phi_j(y)\right| &\leq  
\left| \lambda_j \int_0^1 \left( K(x,z) - K(y,z)\right) 
z \cdot \left(\frac{\phi_j (z)}{z}\right) dz\right| \\ 
 &\leq C_0 \lambda_j^4 \Delta_1 (x,y) ,
\end{align*} 
where $\Delta_1 (x,y)$ is as described in Definitions~2.5. 
By this, together with the upper bound \eqref{Delta1Bound} for $\Delta_1 (x,y)$, 
the lemma follows. 
\end{proof} 
\begin{remarks}{\it 1)} Let $j\in{\mathbb N}$.  
Then, by Lemma~3.4, the eigenfunction $\phi_j(x)$ satisfies 
a uniform Lipschitz condition of order $1$ on each closed interval $[a,b]\subset (0,1]$. 
In particular, for all $\varepsilon > 0$, 
there is some $M_{j,\varepsilon} <\infty$ such that 
\begin{equation*} 
\left| \phi_j(x) - \phi_j(y)\right|\leq M_{j,\varepsilon} |x-y| 
\qquad\text{for all $x,y\in [\varepsilon , 1]$} 
\end{equation*}  
(Lemma~3.4 implies that this holds with $M_{j,\varepsilon} := 4 C_0 \lambda_j^4 \varepsilon^{-2}$). 
It follows that the function $\phi_j(x)$ is absolutely continuous 
on each closed interval $[a,b]\subset (0,1]$, and so is of 
bounded variation on any such interval (this last fact may also be deduced   
directly from Lemma~3.4). 
\item{\it 2)} We will later improve upon Lemma~3.4: see Corollary~4.10. 
\end{remarks} 

\section{The first derivative} 
Let $j\in{\mathbb N}$, and put $\lambda = \lambda_j$ and $\phi (x) = \phi_j (x)$ ($0\leq x\leq 1$). 
We recall (see our Remarks following Lemma~3.4)  
that $\phi(x)$ is of bounded variation 
(and is, moreover, absolutely continuous) on any closed interval $[a,b]\subset (0,1]$.  
\par
It is well-known (see \cite[Sections~11.3--11.42]{Titchmarsh 1939}, for example) 
that any function that is of bounded variation 
on some interval $X$ must be differentiable {\it almost everywhere} 
(with respect to the Lebesgue measure) in that same interval. 
If the function in question is absolutely continuous on $X$, and if $X$ is compact,  
then the derivative of the function is Lebesgue integrable on $X$ (even if  
the set of points at which that derivative is defined is a proper subset of $X$) 
and the function is (on $X$) a Lebesgue indefinite integral of its derivative: 
for proof of this see \cite[Sections~11.4, 11.54, 11.7 and~11.71]{Titchmarsh 1939}. 
By applying these observations to our eigenfunction $\phi(x)$, 
we deduce from what was noted in the preceding paragraph 
that $\phi$ is differentiable almost everywhere in $[0,1] =  
\{ 0\}\cup\left( \cup_{n\in{\mathbb N}} \left[ n^{-1} , 1\right]\right)$, 
that the derivative $\phi'(x)$ is Lebesgue integrable 
on any closed interval $[a,b]\subset (0,1]$, 
and that 
\begin{equation}\label{IndefiniteIntegral}
\phi(1) - \phi(x)=\int_x^1 \phi'(y) \,dy\qquad\hbox{($0<x\leq 1$).}
\end{equation} 
By this and Theorem~2.10, one has:  
\begin{equation}\label{phi(1)asIntegral}
\lim_{x\rightarrow 0+}\int_x^1 \phi'(y) \,dy = \phi(1) - \phi(0) = \phi(1)\;.
\end{equation} 
For more specific information about $\phi'(x)$ we need the following result. 
\begin{theorem} 
The function $x\mapsto x^{-1} \phi(x)$ is Lebesgue integrable on $[0,1]$,  
and so 
\begin{equation}\label{DefPhi1}
{\mathbb R}\ni\int_0^1 \frac{\phi(y) dy}{y} = \Phi_1 \quad\text{(say).}  
\end{equation} 
When $0 < x < 1$ and  $\frac{1}{x}$ is not an integer, one has 
\begin{equation}\label{DerivativeFormula1} 
\lambda^{-1} x^2 \phi'(x) = \Phi_1 - \sum_{m>\frac{1}{x}} \frac{\phi\left(\frac{1}{mx}\right)}{m} \in{\mathbb R}. 
\end{equation} 
For $n\in{\mathbb N}$, the derivative $\phi'(x)$ is a continuous function on 
the interval $\left( (n+1)^{-1} , n^{-1}\right)$, and one has both 
\begin{equation}\label{LefthandDphi} 
\lim_{x\rightarrow \frac{1}{n} -} \phi'(x) 
= \lambda n^2 \left( \Phi_1 - \sum_{m=n+1}^{\infty} \frac{\phi\left(\frac{n}{m}\right)}{m}\right)\in{\mathbb R}  
\end{equation} 
and 
\begin{equation}\label{RighthandDphi} 
\lim_{x\rightarrow \frac{1}{n+1} +} \phi'(x) 
= \lambda (n+1)^2 \left( \Phi_1 - \sum_{m=n+1}^{\infty} \frac{\phi\left(\frac{n+1}{m}\right)}{m}\right)\in{\mathbb R} . 
\end{equation} 
\end{theorem} 
\begin{proof} 
For $n\in{\mathbb N}$ and $0\leq x\leq 1$ we put 
\begin{equation*}
f_n(x) = \begin{cases} 
x^{-1} \phi(x) & \text{if $x\geq n^{-1}$} , \\ 
0 & \text{otherwise} . 
\end{cases}
\end{equation*} 
Since $\phi$ is a measurable function on $[0,1]$, 
it follows that $f_1,f_2,f_3,\ldots $ is a sequence of measurable functions on $[0,1]$. 
Theorem~3.2 implies that this sequence of functions is uniformly bounded. 
Given these facts, and given that the equality 
$\lim_{n\rightarrow\infty} f_n(x) = x^{-1} \phi(x)$ 
holds almost everywhere on $[0,1]$ (everywhere except at $x=0$, in fact), 
it therefore follows by Lebesgue's theorem of bounded convergence 
\cite[Section~10.5]{Titchmarsh 1939} that one has what is 
stated in the first part of the theorem (i.e. up to and including \eqref{DefPhi1}). 
\par
To complete this proof we shall show that \eqref{DerivativeFormula1} holds whenever 
$x$ satisfies the attached conditions. Those conditions imply that, for 
some positive integer $n$, one has 
\begin{equation}\label{ProofC1}
\frac{1}{n+1} < x < \frac{1}{n} .
\end{equation} 
Thus it will be enough to show that, when $n\in{\mathbb N}$, 
one has \eqref{DerivativeFormula1} for all $x$ satisfying \eqref{ProofC1}. 
\par
Let $n\in{\mathbb N}$. Then it follows from \eqref{DefEigenfunction} and \eqref{DefK} that, 
for $x$ satisfying \eqref{ProofC1} and $H\in{\mathbb N}$, one has: 
\begin{align*} 
\lambda^{-1}\phi(x) 
 &= \int_0^{\frac{1}{(n+H+1)x}} K(x,y) \phi(y) dy  + 
 \int_{\frac{1}{(n+1)x}}^1 \left( \frac{1}{2} +  n - \frac{1}{xy} \right) \phi(y) dy \\
 &\phantom{{=}} + \sum_{h=1}^H \int_{\frac{1}{(n+h+1)x}}^{\frac{1}{(n+h)x}} 
\left( \frac{1}{2} +  n + h - \frac{1}{xy} \right) \phi(y) dy\\  
 &=  r_H(x) + u_0(x) + \sum_{h=1}^H u_h(x) \quad\text{(say).} 
\end{align*}
By \eqref{DefK} and the Theorem~3.2, the above term $r_H(x)$ satisfies 
\begin{equation*} 
\left| r_H(x)\right| \leq \int_0^{\frac{1}{(n+H+1)x}} \left| K(x,y) \phi(y)\right| dy
\leq {\textstyle\frac{1}{2}} C_0 |\lambda|^3 \int_0^{\frac{1}{(n+H+1)x}} y dy  
< \frac{C_0 |\lambda|^3}{x^2 H^2} . 
\end{equation*} 
Thus $r_H(x)\rightarrow 0$ as $H\rightarrow\infty$, so that we have  
\begin{equation}\label{ProofC2} 
\lambda^{-1}\phi(x) = 
u_0(x) + \sum_{h=1}^{\infty} u_h(x) , 
\quad\text{when $x$ satisfies \eqref{ProofC1}.} 
\end{equation} 
\par 
We now contemplate term-by-term differentiation of the right-hand side 
of Equation~\eqref{ProofC2}, on the assumption that $x$ satisfies \eqref{ProofC1}. 
But first let us define functions $v_0(x),v_1(x),v_2(x),\ldots $   
on the closed interval $\left[ (n+1)^{-1} , n^{-1}\right]$, by specifying that  
\begin{equation*} 
x^2 v_h(x) = 
\begin{cases} 
{\displaystyle\int_{\frac{1}{(n+1)x}}^1 \frac{\phi(y) dy}{y} 
- \frac{\phi\left(\frac{1}{(n+1)x}\right)}{2(n+1)}} 
 & \text{if $h=0$} , \\ 
{\displaystyle\int_{\frac{1}{(n+h+1)x}}^{\frac{1}{(n+h)x}} \frac{\phi(y) dy}{y} 
- \frac{\phi\left(\frac{1}{(n+h+1)x}\right)}{2(n+h+1)} - 
\frac{\phi\left(\frac{1}{(n+h)x}\right)}{2(n+h)}}  
 & \text{if $h\in{\mathbb N}$} 
\end{cases}
\end{equation*}
(note the function $x\mapsto x^{-1} \phi(x)$ is integrable on $[0,1]$, 
and so is also integrable on all of the ranges of integration occurring here, 
since these ranges are subintervals of $[0,1]$ whenever $x\geq (n+1)^{-1}$).   
Using the part of the theorem that was already proved, we deduce that, when 
$H\in{\mathbb N}$ and $(n+1)^{-1}\leq x\leq n^{-1}$, one has: 
\begin{equation}\label{ProofC3}  
\int_0^{\frac{1}{(n+H+1)x}} \frac{\phi(y) dy}{y} + 
\frac{\phi\left(\frac{1}{(n+H+1)x}\right)}{2(n+H+1)} +   
x^2\sum_{h=0}^H v_h (x) = 
\Phi_1 - \sum_{h=1}^H \frac{\phi\left(\frac{1}{(n+h)x}\right)}{n+h} .
\end{equation} 
\par 
Since $\phi(x)$ is continuous on $(0,1]$, we find that the function 
$v_0(x)$, and each function in the sequence $v_1(x),v_2(x),v_3(x)\ldots $, 
is continuous on the closed interval $\left[ (n+1)^{-1} , n^{-1}\right]$. 
By Theorem~3.2, we  find also that, when $h\in{\mathbb N}$ and 
$(n+1)^{-1}\leq x\leq n^{-1}$, one has  
\begin{align*} 
\left| v_h (x)\right| 
 &\leq \frac{1}{x^2}\int_{\frac{1}{(n+h+1)x}}^{\frac{1}{(n+h)x}} C_0 |\lambda|^3 dy  
+ \frac{C_0 |\lambda|^3}{2(n+h+1)^2 x^3} + \frac{C_0 |\lambda|^3}{2(n+h)^2 x^3} \\
 &= {\textstyle\frac{1}{2}} C_0 \left( \frac{|\lambda|}{x}\right)^3 \left( \frac{1}{n+h+1} + \frac{1}{n+h}\right)^2 
 < \frac{2 C_0 |\lambda|^3 (n+1)^3}{h^2} . 
\end{align*} 
Thus application of the Weierstrass $M$-test \cite[Theorem~9.6]{Apostol 1974}  
shows that the series $v_1(x) + v_2(x) + v_3(x) + \ldots $ is uniformly convergent 
on the interval $\left[ (n+1)^{-1} , n^{-1}\right]$. Therefore, 
given that each of $v_1(x),v_2(x),v_3(x)\ldots $ (and $v_0(x)$ also)  
is continuous on $\left[ (n+1)^{-1} , n^{-1}\right]$, it follows that we have 
\begin{equation}\label{ProofC4} 
\sum_{h=0}^{\infty} v_h (x) = v_0 (x) + \sum_{h=1}^{\infty} v_h (x) = g(x) 
\quad\text{for all $x\in\left[\frac{1}{n+1} , \frac{1}{n}\right]$,} 
\end{equation} 
where $g(x)$ is some continuous real-valued function on $\left[ (n+1)^{-1} , n^{-1}\right]$.  
\par
We observe now that, by Theorem~3.2, the sum of first two terms on 
the left-hand side of Equation~\eqref{ProofC3} is a number $\rho_H(x)$ 
that satisfies  
\begin{equation*} 
\left|\rho_H (x)\right| \\ 
\leq \left( \frac{C_0 |\lambda|^3}{(n+H+1)x}\right)\left( 1 + \frac{1}{2(n+H+1)}\right) .
\end{equation*} 
In particular, for each fixed $x\in\left[ (n+1)^{-1} , n^{-1}\right]$, 
we have $\rho_H (x) \rightarrow 0$ as $H\rightarrow\infty$.
This, together with \eqref{ProofC3} and \eqref{ProofC4}, enables 
us to deduce that, for $(n+1)^{-1}\leq x\leq n^{-1}$, one has 
\begin{equation}\label{ProofC5} 
\sum_{m=n+1}^{\infty} \frac{\phi\left(\frac{1}{m x}\right)}{m} = 
\Phi_1 - x^2\sum_{h=0}^{\infty} v_h (x) = \Phi_1 - x^2 g(x)\in{\mathbb R} .
\end{equation}
\par 
Assuming that \eqref{ProofC1} holds, it follows 
by \eqref{DefK}, Theorem~2.10 and elementary calculus that one has 
\begin{align}\label{ProofC6} 
u_0'(x) 
 &= \frac{d}{dx}\int_{\frac{1}{(n+1)x}}^1 \left( \frac{1}{2} +   
n - \frac{1}{xy} \right) \phi(y) dy \nonumber\\ 
 &= \int_{\frac{1}{(n+1)x}}^1 
\frac{\partial}{\partial z}\left( \left( \frac{1}{2} + n - \frac{1}{zy} \right) \phi(y)\right) 
\biggr|_{z=x}  dy \nonumber\\ 
 &\phantom{{=}} - \left( \frac{1}{2} +  n - \frac{1}{xy} \right) \phi(y)\biggr|_{y=\frac{1}{(n+1)x}}  
\cdot\frac{d}{dx}\left( \frac{1}{(n+1)x}\right) \nonumber\\ 
 &=\frac{1}{x^2} \int_{\frac{1}{(n+1)x}}^1 \frac{\phi(y) dy}{y} 
+ {\textstyle\frac{1}{2}} \phi\left(\frac{1}{(n+1)x}\right) \cdot \frac{(-1)}{(n+1) x^2} \nonumber\\ 
 &= v_0(x) . 
\end{align} 
Similarly, for $h\in{\mathbb N}$, we find (subject to \eqref{ProofC1} holding) that 
\begin{align}\label{ProofC7} 
u_h'(x) &= \frac{d}{dx} \int_{\frac{1}{(n+h+1)x}}^{\frac{1}{(n+h)x}} 
\left( \frac{1}{2} +  n + h - \frac{1}{xy} \right) \phi(y) dy \nonumber\\ 
 &= \frac{1}{x^2} \int_{\frac{1}{(n+h+1)x}}^{\frac{1}{(n+h)x}} \frac{\phi(y) dy}{y} 
+ {\textstyle\frac{1}{2}} \phi\left(\frac{1}{(n+h+1)x}\right) \cdot \frac{(-1)}{(n+h+1) x^2} \nonumber\\ 
 &\phantom{{=}} + {\textstyle\frac{1}{2}} \phi\left(\frac{1}{(n+h)x}\right) \cdot \frac{(-1)}{(n+h) x^2} \nonumber\\ 
 &= v_h (x) . 
\end{align} 
\par
In preparation for the next steps, we now recall and process 
certain pertinent facts 
that have already been established. 
\par 
We have seen that the functions 
$u_0(x),u_1(x),u_2(x),\ldots $ (defined, implicitly, a few lines above \eqref{ProofC2}) 
are real-valued, and are defined on the interval $\left( (n+1)^{-1} , n^{-1}\right)$. 
We found that, at all points $x$ of the same open interval, 
the series $u_0(x) + u_1(x) + u_2(x) + \ldots $ is convergent  
and the derivatives $u_0'(x),u_1'(x),u_2'(x),\ldots $ exist and are finite 
(their values were computed in \eqref{ProofC6} and \eqref{ProofC7}).   
Moreover, since the series $v_1(x) + v_2(x) + v_3(x) + \ldots $ was 
found to be uniformly convergent on $\left[ (n+1)^{-1} , n^{-1}\right]$, 
and since we have (by $\eqref{ProofC7}$) $u_h'(x) = v_h(x)$ 
whenever $(n+1)^{-1} < x < n^{-1}$ and $h\in{\mathbb N}$, we may make the (trivial) deductions   
that the series $u_1'(x) + u_2'(x) + u_3'(x) + \ldots $ is uniformly convergent on 
$\left( (n+1)^{-1} , n^{-1}\right)$, and that the same may therefore be said of   
the series $u_0'(x) + u_1'(x) + u_2'(x) + \ldots $.  
\par
Given the fact just noted (in the last paragraph), it follows by 
\cite[Theorem~9.14]{Apostol 1974} that the function $x\mapsto \sum_{h=0}^{\infty} u_h(x)$ 
is differentiable at all points of the interval $\left( (n+1)^{-1} , n^{-1}\right)$, 
and that one has: 
\begin{equation}\label{ProofC8} 
\frac{d}{dx} \sum_{h=0}^{\infty} u_h (x) = 
\sum_{h=0}^{\infty} u_h' (x) 
\quad\text{when $x$ satisfies \eqref{ProofC1}} . 
\end{equation} 

Subject to \eqref{ProofC1} holding, it follows by \eqref{ProofC2}, \eqref{ProofC8}, 
\eqref{ProofC6}, \eqref{ProofC7} and \eqref{ProofC4} that $\phi'(x)$ exists, and that one has 
\begin{equation}\label{ProofC9}  
\lambda^{-1} \phi'(x) = 
\sum_{h=0}^{\infty} u_h'(x) = 
\sum_{h=0}^{\infty} v_h(x) = g(x) . 
\end{equation} 
We recall that the function $g(x)$ was shown to be continuous on 
the closed interval $\left[ (n+1)^{-1} , n^{-1}\right]$. 
Thus it is a corollary of \eqref{ProofC9}  that the derivative $\phi'(x)$ 
is a continuous function on $\left( (n+1)^{-1} , n^{-1}\right)$, and that one has: 
\begin{equation}\label{ProofC10} 
\lim_{x\rightarrow \frac{1}{n} -} \phi'(x) = \lambda g\left(\frac{1}{n}\right) 
\quad\text{and}\quad 
\lim_{x\rightarrow \frac{1}{n+1} +} \phi'(x) = \lambda g\left(\frac{1}{n+1}\right)  . 
\end{equation} 
With the help of \eqref{ProofC5}, we deduce from \eqref{ProofC10} and \eqref{ProofC9} 
what is stated in \eqref{LefthandDphi} and \eqref{RighthandDphi}, and also    
the cases of \eqref{DerivativeFormula1} in which $x$ satisfies \eqref{ProofC1}. 
This (as explained earlier) completes our proof of the theorem.
\end{proof} 
\begin{corollary} 
When $n\in{\mathbb N}$, the  restriction of $\phi(x)$ to the 
closed interval $\left[ (n+1)^{-1} , n^{-1}\right]$ is continuously differentiable on $\left[ (n+1)^{-1} , n^{-1}\right]$. 
\end{corollary} 
\begin{proof} 
Let $n\in{\mathbb N}$, $a=\frac{1}{n+1}$ and $b=\frac{1}{n}$. 
Let $\rho(x)$ is the restriction of $\phi(x)$ to the interval $[a,b]$. 
\par 
Suppose, firstly, that $a<y\leq b$. Then one has 
$\frac{\rho(y) -\rho(a)}{y-a} = \frac{\phi(y)-\phi(a)}{y-a}$, 
and so, since $\phi(x)$ is continuous on $[0,1]\supset [a,b] \supseteq [a,y]$, 
and is differentiable on $(a,b)\supseteq (a,y)$, 
it follows by the mean value theorem of differential calculus that, 
for some $c\in (a,y)$,  one has: 
\begin{equation}\label{ProofD1} 
\frac{\rho(y) -\rho(a)}{y-a} = \phi'(c) .
\end{equation} 
Since we have here $a<c<y$, it follows that $c\rightarrow a+$ as $y\rightarrow a+$, 
and so it may be deduced from \eqref{ProofD1} and \eqref{RighthandDphi} that one has: 
\begin{equation}\label{ProofD2} 
\rho'(a) := \lim_{y\rightarrow a+} \frac{\rho(y) -\rho(a)}{y-a} = 
\lim_{c\rightarrow a+} \phi'(c) \in{\mathbb R} . 
\end{equation} 
Using instead \eqref{LefthandDphi}, one can show (similarly) that 
\begin{equation}\label{ProofD3} 
\rho'(b) := \lim_{y\rightarrow b-} \frac{\rho(b) -\rho(y)}{b-y} = 
\lim_{c\rightarrow b-} \phi'(c) \in{\mathbb R} . 
\end{equation} 
When $a<z<b$, one has 
\begin{equation}\label{ProofD4} 
\frac{\rho(y) -\rho(z)}{y-z} = \frac{\phi(y)-\phi(z)}{y-z} 
\quad \text{for all $y\in [a,z)\cup (z,b]$, } 
\end{equation} 
and so (given that $\phi'(z)$ exists and is finite, by virtue of $\phi'(x)$ being 
continuous on $(a,b)$) one finds, by taking the limit as $y\rightarrow z$ of 
both sides of \eqref{ProofD4}, that $\rho'(z) = \phi'(z)\in{\mathbb R}$ for 
$a<z<b$. Thus $\rho'(x)$ is continuous on $(a,b)$ (since $\phi'(x)$ is), 
and $\rho'(c)$ may be substituted for $\phi'(c)$ in both 
\eqref{ProofD2} and \eqref{ProofD3}, so enabling us to conclude that 
$\rho'(x)$ is also continuous at the points $x=a$ and $x=b$. 
The derivative $\rho'(x)$ is therefore continuous on $[a,b]$. 
\end{proof} 
\begin{corollary} 
The function $\phi(x)$ is continuously differentiable on $\bigl(\frac{1}{2} , 1\bigr]$. One has  
\begin{equation}\label{DphiAt1}
{\mathbb R}\ni\phi'(1) = \lambda \Phi_1 - \lambda\sum_{m=2}^{\infty} \frac{\phi\left(\frac{1}{m}\right)}{m} , 
\end{equation} 
and also: 
\begin{equation}\label{DphiJump} 
\phi_{+}'\left( \frac{1}{n}\right) - \phi_{-}'\left( \frac{1}{n}\right) = 
-\lambda \phi(1) n 
\quad\text{for $n=2,3,4,\ldots $ ,} 
\end{equation} 
where $\phi_{+}'(x)$ and $\phi_{-}'(x)$ are, respectively,  the righthand and lefthand derivatives of $\phi(x)$ 
(so that $\phi_{\pm}'(x) := \lim_{y\rightarrow x\pm} \frac{\phi(y) - \phi(x)}{y-x}$). 
\par 
If $\phi(1) \neq 0$ then $\left\{ \frac{1}{2} , \frac{1}{3} , \frac{1}{4} , \ldots \right\}$ is the set of points of the interval $(0,1]$ at which $\phi(x)$ is not differentiable. 
\par 
If $\phi(1)=0$ then $\phi(x)$ is continuously differentiable on $(0,1]$, and 
\eqref{DerivativeFormula1} holds for all $x\in (0,1]$. 
\end{corollary} 
\begin{proof}
Since the domain of $\phi(x)$ contains no number greater than $1$, it follows from the case $n=1$ of the preceding corollary that one has ${\mathbb R}\ni \phi'(1) = \lim_{x\rightarrow 1-}\phi'(x)$, so that $\phi'(x)$ is continuous at the point $x=1$. By this, together with the case $n=1$ of \eqref{LefthandDphi}, one obtains the result \eqref{DphiAt1}. Since we know (by Theorem~4.1) that $\phi'(x)$ is continuous on $\left(\frac{1}{2}, 1\right)$, 
and have just found $\phi'(x)$ to be continuous at $x=1$, 
it therefore follows (trivially) that $\phi(x)$ is continuously differentiable on $\bigl(\frac{1}{2}, 1\bigr]$. 
\par
By Corollary~4.2 again (not only in the form stated, but also with $n-1$ substituted for $n$) we find that, 
for either (consistent) choice of sign ($\pm$), one has: 
\begin{equation}\label{ProofE1} 
\phi_{\pm}'\left( \frac{1}{n}\right) = \lim_{x\rightarrow\frac{1}{n}\pm} \phi'(x) 
\quad\text{for $n=2,3,4,\ldots $ . }
\end{equation} 
The combination of \eqref{ProofE1},  \eqref{RighthandDphi} (with $n-1$ substituted for $n$) and \eqref{LefthandDphi}, yields (immediately) the result stated in \eqref{DphiJump}. 
\par 
Theorem~4.1 tells us that $\phi(x)$ is differentiable on each one of the open intervals  
$\left(\frac12 ,1\right) , \left(\frac13 ,\frac12\right) , \left(\frac14 ,\frac13\right) ,\,\ldots$~, and so 
(recalling \eqref{DphiAt1}) we may conclude that 
the set $\left\{ \frac12 , \frac13 , \frac14 ,\,\ldots\,\right\}$ contains all points of the interval $(0,1]$ at which $\phi(x)$ is not differentiable. 
If $\phi(1)\neq 0$ then, by \eqref{DphiJump}, it follows that, for $n=2,3,4,\ldots $ , we have 
$\phi_{+}'(1/n) \neq \phi_{-}'(1/n)$. 
Thus $\frac{1}{2}, \frac{1}{3}, \frac{1}{4}, \ldots $ are  (in this case) points at which $\phi(x)$ is not differentiable.
\par 
Suppose that one has instead $\phi(1)=0$, then \eqref{DphiJump} gives $\phi_{+}'(1/n) = \phi_{-}'(1/n)$, for $n=2,3,4,\ldots $. 
Thus $\phi(x)$ is (in the case being considered)  differentiable at 
every  point of the set $\left\{ \frac12 , \frac13 , \frac14 ,\,\ldots\,\right\}$. 
By this, combined with the first of our conclusions in the preceding paragraph, 
it follows that $\phi(x)$ is differentiable on $(0,1]$. By this and Corollary~4.2, 
one  may deduce that, for each $n\in{\mathbb N}$, the restriction of $\phi'(x)$ to 
the interval $\left[(n+1)^{-1} , n^{-1}\right]$ is continuous on that same interval. 
Therefore, given that each point in the sequence $\frac{1}{2}, \frac{1}{3}, \frac{1}{4}, \ldots\,$  
is a left hand boundary point of one of the intervals in the sequence 
$\left[\frac12 , 1\right] , \left[\frac13 , \frac12\right] , \left[\frac14 , \frac13\right] , \,\ldots $~, and is 
(at the same time) a right hand boundary point of another interval from the same sequence, we may conclude that the continuity of the restrictions of $\phi'(x)$ to each of those intervals implies the continuity of $\phi'(x)$ at each point in the sequence $\frac{1}{2}, \frac{1}{3}, \frac{1}{4}, \,\ldots $~. By this and the relevant result stated in Theorem~4.1, we find that $\phi'(x)$ is continuous on $(0,1)$. We showed (above) that, regardless of whether or not $\phi(1)=0$, the function $\phi'(x)$ is continuous at $x=1$. Thus we may now conclude that $\phi'(x)$ is continuous on $(0,1)\cup\{ 1\} = (0,1]$, provided that $\phi(1)$ equals $0$; moreover $\phi'(x)$ is then continuous at each point in the sequence $1, \frac{1}{2}, \frac{1}{3},\,\ldots $~, and so it follows by \eqref{LefthandDphi} that one has  \eqref{DerivativeFormula1} for all values of $x$ in that sequence; we also know (from Theorem~4.1) that \eqref{DerivativeFormula1} holds at all points of the interval $(0,1]$ that are not terms of the sequence just mentioned: we conclude that, if $\phi(1)$ equals $0$, then \eqref{DerivativeFormula1} holds for all $x\in(0,1]$. 
\end{proof} 
\begin{lemma} The definite integral $\Phi_1$ that is defined in \eqref{DefPhi1} satisfies 
\begin{equation*} 
\left| \Phi_1\right| < {\textstyle\frac{3}{2}} | \lambda | . 
\end{equation*} 
\end{lemma} 
\begin{proof} 
Let $C_0$ be the constant defined in \eqref{DefC0}, and put 
$\Delta := (2 C_0)^{-2/3} \lambda^{-2}$. 
Then, since $C_0 > \frac{1}{3}$, it 
follows by \eqref{SpectrumBound} that we have $2C_0 |\lambda |^3 > \frac{16}{3}$, 
and so $0 < \Delta < \left( \frac{3}{16}\right)^{2/3} < 1$. Therefore, with the 
help of the Cauchy-Schwarz ineqality,  we obtain: 
\begin{align*} 
\left| \Phi_1\right| = \biggl| \int_0^1 \frac{\phi(y) dy}{y} \biggr| 
 &\leq \int_0^{\Delta} \frac{|\phi(y)| dy}{y} + 
\left( \int_{\Delta}^1 \frac{dy}{y^2}\right)^{1/2} \| \phi \| \\ 
 &= \int_0^{\Delta} \frac{|\phi(y)| dy}{y} + \left( \frac{1}{\Delta} - 1\right)^{1/2} \cdot 1 . 
\end{align*} 
We use Theorem~3.2 to bound the last of the integrals here, and so find that 
$\left| \Phi_1\right| < \Delta C_0 |\lambda|^3 + \Delta^{-1/2} 
= \left( 2^{-2/3} + 2^{1/3}\right) C_0^{1/3} |\lambda| = 
\frac{3}{2} (2 C_0)^{1/3} |\lambda| < \frac{3}{2} |\lambda|$. 
\end{proof} 
\begin{lemma} 
For $0<x\leq 1$, one has 
\begin{equation*} 
\Biggl| \sum_{m>\frac{1}{x}} \frac{\phi\left(\frac{1}{mx}\right)}{m} \Biggr| < 
\left( {\textstyle\frac{3}{2}} + \log |\lambda |\right) |\lambda | . 
\end{equation*} 
\end{lemma} 
\begin{proof} 
We begin similarly to the proof of Lemma~4.4, but now put instead 
$\Delta := (4C_0)^{-1} \lambda^{-2}$, so that 
$0<\Delta < \frac{3}{4} \lambda^{-2} \leq \frac{3}{16} < 1$. 
Let $0<x\leq 1$. By \eqref{PhiBoundedUniformly}  and Theorem~3.2, one has: 
\begin{align*} 
\sum_{m>\frac{1}{x}} \frac{\left| \phi\left(\frac{1}{mx}\right)\right| }{m}
 &\leq \sum_{\frac{1}{x} < m \leq \frac{1}{\Delta x}} \frac{|\lambda|}{2 m} + 
 \sum_{m>\frac{1}{\Delta x}} \frac{C_0 |\lambda|^3}{m^2 x} \\ 
 &< {\textstyle\frac{1}{2}} |\lambda| \left( \int_{\frac{1}{x}}^{\frac{1}{\Delta x}} \frac{dy}{y} + x\right) + 
C_0 |\lambda|^3 x^{-1} \left( \int_{\frac{1}{\Delta x}}^{\infty} \frac{dy}{y^2} + (\Delta x)^2\right) \\ 
 &= {\textstyle\frac{1}{2}} |\lambda| \left( x + \log\left(\frac{1}{\Delta}\right)\right) 
 + C_0 |\lambda|^3 x^{-1}\left( \Delta x + (\Delta x)^2\right) \\ 
 &< {\textstyle\frac{1}{2}} |\lambda| \left( 1 + \log\left(\frac{1}{\Delta}\right)\right) 
 + 2 C_0 |\lambda|^3 \Delta = 
{\textstyle\frac{1}{2}} |\lambda| \left( 2 + \log\left(4 C_0 \lambda^2\right)\right) .
\end{align*} 
Since $4 C_0 < 2 < e$, the desired bound follows.
\end{proof} 
\begin{theorem} Let $0<x\leq 1$. If $\phi'(x)$ exists, then it satisfies 
\begin{equation*} 
\left| \phi'(x)\right| < \frac{\left( 3 + \log |\lambda |\right) \lambda^2}{x^2} . 
\end{equation*} 
\end{theorem} 
\begin{proof} 
Suppose that $\phi'(x)$ exists. Then, by Corollary~4.3 and Theorem~4.1, it follows 
that $\phi'(x)$ is given by the equation \eqref{DerivativeFormula1}. 
By \eqref{DerivativeFormula1} and Lemmas~4.4 and~4.5, it follows that one has    
$|\lambda^{-1} x^2 \phi'(x)| 
\leq \frac{3}{2} |\lambda| + \left( \frac{3}{2} + \log |\lambda|\right) |\lambda|$. 
\end{proof} 
\begin{lemma} 
Let $0<x\leq 1$, and let $C_0$ be the positive constant given by \eqref{DefC0}. 
Suppose that $\phi'(x)$ exists, and that $0<\Delta <1$. 
Then one has 
\begin{equation*} 
\lambda^{-1} x \phi'(x) = 
\int_{\Delta}^1 y\phi(y) dK(x,y) + E_1 , 
\end{equation*} 
for some real number $E_1=E_1(\phi ;x,\Delta)$ that satisfies:  
\begin{equation*} 
\left| E_1\right| \leq \frac{3 C_0 |\lambda|^3 \Delta}{x} . 
\end{equation*} 
\end{lemma} 
\begin{proof} 
Using the definition of $K(x,y)$, given in \eqref{DefK}, 
we obtain the following reformulation of the above Riemann-Stieltjes integral:  
\begin{align}\label{ProofF1} 
\int_{\Delta}^1 y \phi(y) dK(x,y) 
 &= \int_{\Delta}^{1} y\phi(y) d\left( \left\lfloor\frac{1}{xy}\right\rfloor - \frac{1}{xy}\right) \nonumber\\ 
 &= \int_{\Delta}^{1} y\phi(y) d\left\lfloor\frac{1}{xy}\right\rfloor - 
 \int_{\Delta}^{1} y\phi(y) d\left( \frac{1}{xy}\right) \nonumber\\
 &=\sum_{\frac{1}{x} < m\leq \frac{1}{\Delta x}} (-1)\left( \frac{1}{xm}\right) \phi\left( \frac{1}{xm}\right)  - 
 \int_{\Delta}^1 y\phi(y)\left( - \frac{1}{xy^2}\right) dy \nonumber\\ 
 &=\frac{1}{x} \cdot \Biggl(\int_{\Delta}^1 \frac{\phi(y) dy}{y} - 
  \sum_{\frac{1}{x} < m\leq \frac{1}{\Delta x}} \frac{\phi\left( \frac{1}{xm}\right)}{m} \Biggr) \in{\mathbb R} 
\end{align}
(note that Theorem~2.10 justifies all of these steps, since it implies 
that the integrands $y\phi(y)$ and $y^{-1}\phi(y)$ are continuous on $[\Delta , 1]$). 
Here (as in the proof of Theorem~4.6) we may apply \eqref{DerivativeFormula1}: using that 
result, and also \eqref{DefPhi1}, we deduce from \eqref{ProofF1} that one has 
\begin{equation*} 
\int_{\Delta}^1 y \phi(y) dK(x,y) - \lambda^{-1} x \phi'(x) =  E_1 , 
\end{equation*} 
where  
\begin{align*} 
{\mathbb R}\ni E_1=E_1(\phi ; x, \Delta) &:= \frac{1}{x} \cdot \Biggl(\int_{\Delta}^1 \frac{\phi(y) dy}{y} - \Phi_1 + 
  \sum_{m > \frac{1}{\Delta x}} \frac{\phi\left( \frac{1}{xm}\right)}{m} \Biggr) \\ 
 &= \frac{1}{x} \cdot \Biggl( \sum_{m > \frac{1}{\Delta x}} \frac{\phi\left( \frac{1}{xm}\right)}{m} - 
 \int_0^{\Delta} \frac{\phi(y) dy}{y} \Biggr) . 
\end{align*} 
As seen earlier (in the proofs of Lemmas~4.4 and~4.5), we have here both 
\begin{equation*} 
\Biggl| \int_0^{\Delta} \frac{\phi(y) dy}{y} \Biggr| \leq C_0 |\lambda |^3 \Delta 
\quad\text{and}\quad 
\Biggl| \sum_{m > \frac{1}{\Delta x}} \frac{\phi\left( \frac{1}{xm}\right)}{m} \Biggr| < 
2C_0 |\lambda |^3 \Delta , 
\end{equation*} 
and so may deduce the desired upper bound on $|E_1|$. 
\end{proof} 
\begin{remarks} 
The kernel $K(x,y)$ is, by \eqref{DefK}, a function on $[0,1]\times[0,1]$ of 
the form $(x,y)\mapsto f(xy)$, where $f(t)$ is a certain real-valued 
function on $[0,1]$ that has discontinuities at the points 
$\frac{1}{2},\frac{1}{3},\frac{1}{4},\ldots $ , and at the point $t=0$. 
If, instead of \eqref{DefK}, we had $K(x,y)=g(xy)$ ($0\leq x,y\leq 1$), where 
$g(t)$ was some real-valued function that was continuously differentiable on $[0,1]$, 
then it could be argued that \eqref{DefEigenfunction} would imply  that 
\begin{align*} 
\lambda^{-1} \phi'(x) 
& = \frac{d}{dx} \int_0^1 g(xy) \phi(y) dy \\ 
&= \int_0^1 \phi(y) \left( \frac{\partial}{\partial x} \left( g(xy)\right)\right) dy \\ 
&= \int_0^1 \phi(y) y g'(xy)  dy \\ 
&= \int_0^1 \phi(y) y x^{-1} \left( \frac{\partial}{\partial z} g(xz)\right)\biggr|_{z=y}  dy  
= \frac{1}{x} \int_0^1 y \phi(y) dg(xy) , 
\end{align*} 
and so we could conclude that $\lambda^{-1} x\phi'(x) = \int_0^1 y \phi(y) dK(x,y)$ for $0<x\leq 1$. 
As things stand (i.e. with $K$ as defined in \eqref{DefK}), the 
above argument lacks validity: yet Lemma~4.7 does get us to within 
touching distance of the same conclusion, since it implies that  
whenever $\phi'(x)$ exists one has 
\begin{equation*} 
\lambda^{-1} x\phi'(x) = \lim_{\Delta\rightarrow 0+} \int_{\Delta}^1 y \phi(y) dK(x,y) .  
\end{equation*} 
\end{remarks} 
\begin{lemma} 
Let $0<x\leq 1$, and let $C_0$ be the positive constant given by \eqref{DefC0}. 
Suppose that $\phi'(x)$ exists, and that $0<\Delta <1$. 
Then one has 
\begin{equation*} 
\lambda^{-1} x \phi'(x) = 
\phi(1) K(x,1) - \lambda^{-1} \phi(x) - \int_{\Delta}^1 K(x,y) y \phi'(y) dy + E_2 ,
\end{equation*} 
for some real number $E_2=E_2(\phi ;x,\Delta)$ that satisfies:  
\begin{equation*} 
\left| E_2\right| \leq \frac{4 C_0 |\lambda|^3 \Delta}{x} . 
\end{equation*} 
\end{lemma} 
\begin{proof} 
The hypotheses permit the application of Lemma~4.7: 
by applying integration by parts \cite[Theorem~7.6]{Apostol 1974} to the 
integral that occurs in that lemma, we find that one has 
\begin{equation}\label{ProofG1} 
\lambda^{-1} x \phi'(x) = 
1 \phi(1) K(x,1) - \Delta \phi(\Delta) K(x,\Delta ) - \int_{\Delta}^1 K(x,y) d\left( y \phi(y)\right) + E_1 , 
\end{equation} 
where $E_1$ is as stated in Lemma~4.7. For any given $M\in{\mathbb N}$, 
one has here 
\begin{multline}\label{ProofG2} 
\int_{\Delta}^1 K(x,y) d\left( y \phi(y)\right) \\ 
= \int_{\Delta}^{\frac{1}{M}} K(x,y) d\left( y \phi(y)\right) 
+ \sum_{1<m\leq M} \int_{\frac{1}{m}}^{\frac{1}{m-1}} K(x,y) d\left( y \phi(y)\right) . 
\end{multline} 
We choose to apply this in the case where $M=\lceil 1/\Delta \rceil -1$. 
In this case we have $M<1/\Delta \leq M+1$, so that $\frac{1}{M+1} \leq \Delta < \frac{1}{M}$:  
note also that, since $M\in{\mathbb Z}$ and $M+1\geq 1/\Delta > 1$, we do indeed have $M\in{\mathbb N}$. 
It follows that every range of integration occurring on the right hand side of \eqref{ProofG2} 
is a non-empty subinterval of some interval in the sequence 
$\left[\frac12 , 1\right] , \left[\frac13 , \frac12\right] , \left[\frac14 , \frac13\right] , \,\ldots\,$~: 
we have, in particular, $\bigl[ \frac{1}{M+1} , \frac{1}{M}\bigr] \supseteq \bigl[ \Delta , \frac{1}{M}\bigr] \neq \emptyset$. Thus, by virtue of Corollary~4.2, it can be  
deduced from \eqref{ProofG2} that one has  
\begin{align*}  
\int_{\Delta}^1 K(x,y) d\left( y\phi(y) \right) 
 &= \int_{\Delta}^{\frac{1}{M}} K(x,y) \left( \phi(y) + y \phi'(y) \right) dy  \\ 
 &\phantom{{=}} + \sum_{1<m\leq M} \int_{\frac{1}{m}}^{\frac{1}{m-1}} K(x,y) \left( \phi(y) + y \phi'(y) \right) dy  \\ 
 &=\int_{\Delta}^1 K(x,y) \phi(y) dy + \int_{\Delta}^1 K(x,y) y \phi'(y) dy . 
\end{align*} 
By this, together with \eqref{ProofG1} and \eqref{DefEigenfunction}, we find that the 
equality stated in the lemma is satisfied when one has  
\begin{equation}\label{ProofG3} 
E_2 = - \Delta \phi(\Delta) K(x,\Delta ) + \int_0^{\Delta} K(x,y) \phi(y) dy + E_1 . 
\end{equation} 
By Lemma~4.7, \eqref{DefK} and Theorem~3.2, 
we have here:  
\begin{equation*} 
|E_1|\leq 3C_0 |\lambda|^3 \Delta x^{-1} , 
\quad 
\left| \Delta \phi(\Delta) K(x,\Delta )\right| \leq {\textstyle\frac{1}{2}} \Delta \cdot C_0 |\lambda|^3 \Delta 
= {\textstyle\frac{1}{2}} C_0 |\lambda|^3 \Delta^2 
\end{equation*} 
and 
\begin{equation*} 
\int_0^{\Delta} \left| K(x,y) \phi(y) \right| dy \leq 
{\textstyle\frac{1}{2}} C_0 |\lambda|^3 \int_0^{\Delta} y dy = {\textstyle\frac{1}{4}} C_0 |\lambda|^3\Delta^2 . 
\end{equation*} 
Since we have also $\Delta < 1\leq x^{-1}$, the desired upper bound for $|E_2|$ 
follows from the last three bounds above and \eqref{ProofG3}. 
\end{proof} 
\begin{theorem} 
For all $x\in(0,1]$ such that $\phi'(x)$ exists, one has:
\begin{equation*} 
x \phi'(x) = O\left( |\lambda |^3 \log^2 |\lambda | \right) , 
\end{equation*} 
where the implicit constant is absolute. 
\end{theorem} 
\begin{proof}
Let $0<x\leq 1$. Suppose that $\phi'(x)$ exists. 
Then, by Lemma~4.8, 
\eqref{DefK}, \eqref{PhiBoundedUniformly}, \eqref{SpectrumBound} and the triangle inequality, 
one may deduce that 
\begin{equation}\label{ProofH1} 
\left| x \phi'(x) \right| \leq 
{\textstyle\frac{1}{2}} \lambda^2  +  
{\textstyle\frac{1}{2}} |\lambda | \int_{\Delta}^1 \left| y \phi'(y)\right| dy + 
O\left( \lambda^4 x^{-1} \Delta\right) 
\end{equation} 
for $0<\Delta <1$.  
By Theorem~4.6 and \eqref{SpectrumBound}, we have here 
\begin{equation*} 
\int_{\Delta}^1 \left| y \phi'(y)\right| dy \leq 
\left( 3 + \log |\lambda|\right) \lambda^2 
\int_{\Delta}^1 \frac{dy}{y} = 
O\left(  \left( \lambda^2 \log |\lambda| \right) \log\left( \frac{1}{\Delta}\right) \right) . 
\end{equation*} 
Thus we obtain, in particular, 
\begin{equation*} 
\left| x \phi'(x) \right| \leq 
{\textstyle\frac{1}{2}} \lambda^2  +  
O\left(  \left( |\lambda |^3 \log |\lambda| \right) \log\left( \frac{1}{\Delta}\right) \right) + 
O\left( \lambda^4 x^{-1} \Delta\right)  
\end{equation*} 
when $\Delta = |\lambda |^{-1} x$ (for, by \eqref{SpectrumBound}, 
one does have $0<\Delta < 1$ in this case). 
This gives us:  
\begin{align}\label{ProofH2} 
\left| x \phi'(x) \right| &= O\left( \lambda^2 + 
\left( |\lambda |^3 \log |\lambda| \right) \log\left( |\lambda| / x\right) + |\lambda |^3 \right) \nonumber\\ 
&=O\left( \left( \log\left( \frac{1}{x}\right) + \log |\lambda| \right)
|\lambda |^3 \log |\lambda | \right) . 
\end{align} 
\par
We now repeat, with one change, the steps that led to \eqref{ProofH2}. 
The change that we make is to apply \eqref{ProofH2}, instead of 
Theorem~4.6, to that part of the integral $\int_0^1 \left| y \phi'(y) dy\right| dy$ 
where $y<|\lambda|^{-1}$: note that we still put $\Delta = |\lambda |^{-1} x$, and 
so (given that $|\lambda|\geq 2$ and $\log(1/x)\geq 0$) will have 
$1 > |\lambda|^{-1} \geq \Delta$. We find that one has 
\begin{align*} 
\int_{\Delta}^1 \left| y \phi'(y)\right| dy 
 &\leq 
O\left( \lambda^2 \log |\lambda|\right) \cdot  
\int_{\frac{1}{|\lambda|}}^1 \frac{dy}{y} + 
O\left( |\lambda |^3 \log |\lambda|\right)\cdot \int_{\Delta}^{\frac{1}{|\lambda|}} 
\log\left( \frac{1}{y}\right) dy \\ 
&\leq O\left( |\lambda|^3 \log |\lambda|\right) \cdot 
\left( \frac{\log |\lambda |}{|\lambda|} + 
\int_0^{\frac{1}{|\lambda|}} \log\left( \frac{1}{y}\right) dy \right) \\ 
 &= O\left( |\lambda|^3 \log |\lambda|\right) \cdot 
\left( \frac{1 + 2\log |\lambda |}{|\lambda|} \right) = O\left( \lambda^2 \log^2 |\lambda|\right) . 
\end{align*} 
By means of this last estimate and the case $\Delta = |\lambda |^{-1} x$ of \eqref{ProofH1}, 
one finds that the desired bound for $|x\phi'(x)|$ is obtained. 
\end{proof} 
\begin{corollary} 
Let $j\in{\mathbb N}$ and $0<y\leq x\leq 1$. Then one has both 
\begin{equation}\label{LogarithmicBoundOnIncrements} 
\left| \phi_j (x) - \phi_j (y)\right| =  
O\left( \left|\lambda_j\right|^3  \left( \log \left| \lambda_j\right|\right)^2 \cdot 
\log\left(\frac{x}{y}\right) \right) 
\end{equation} 
and 
\begin{equation}\label{ImprovesCor2.14} 
\left| \phi_j (x) - \phi_j (y)\right| \leq 
\left( 3 + \log\left| \lambda_j\right|\right) \lambda_j^2 \cdot\left(\frac{1}{y} - \frac{1}{x}\right) . 
\end{equation} 
\end{corollary} 
\begin{proof} 
By \eqref{IndefiniteIntegral} (applied twice), we have: 
\begin{equation}\label{ProofR1}  
\left| \phi_j (x) - \phi_j (y)\right| = \left| \int_y^x \phi_j' (z) dz\right| 
\leq \int_y^x \left| \phi_j' (z)\right| dz . 
\end{equation} 
The results \eqref{LogarithmicBoundOnIncrements} and \eqref{ImprovesCor2.14}  
follow by combining \eqref{ProofR1} with Theorems~4.9 and~4.6, respectively. 
\end{proof} 
\begin{theorem} 
The function $x\mapsto x \phi'(x)$ (defined almost everywhere in $[0,1]$) 
is both measurable and square integrable on $[0,1]$. One has 
\begin{equation*} 
\int_0^1 \left( x \phi'(x)\right)^2 dx = O\left(  |\lambda |^5 \log^3 |\lambda| \right) . 
\end{equation*} 
\end{theorem} 
\begin{proof} 
By Corollary~4.2 (or Corollary~4.3), the set of points of the interval $[0,1]$ 
at which $\phi'(x)$ is not defined is a set that is countable, and so 
has Lebesgue measure $0$. Note that Corollary~4.2 implies also 
that $\phi'(x)$ is finite at all those points of the interval $(0,1]$ where it exists. 
We may therefore conclude that the functions $\phi'(x)$ and $x\mapsto x \phi'(x)$ 
are each defined almost everywhere in $[0,1]$, and that the latter is finite (and so real-valued) at 
all points where it is defined. 
\par
For $x\in [0,1]$ and $n\in{\mathbb N}$, put 
$f_n(x) := 
(n+1)\phi\left( \frac{nx+1}{n+1}\right) - (n+1)\phi\left( \frac{nx}{n+1}\right)$. 
By Theorem~2.10, the functions $f_1(x),f_2(x),f_3(x),\ldots $ are continuous on 
$[0,1]$, and are therefore measurable on $[0,1]$. 
Since $\phi'(x)$ is defined almost everywhere in $[0,1]$, it can be deduced  
(from the usual definition of $\phi'(x)$ as a limit) 
that we have $\lim_{n\rightarrow\infty} f_n (x) = \phi'(x)$ almost everywhere in $[0,1]$.   
We may conclude from this that, 
since all terms of the sequence $f_1(x),f_2(x),f_3(x),\ldots $ are measurable on $[0,1]$, 
so too is $\phi'(x)$: see, for example, \cite[Theorem~4.12]{Wheeden and Zygmund 2015} regarding this point. 
\par 
Since $\phi'(x)$ is measurable on $[0,1]$, so is its product with any other such 
function: the functions $x\mapsto x\phi'(x)$ and $x\mapsto \left( x\phi'(x)\right)^2$, 
in particular, are measurable on $[0,1]$. By Theorem~4.9, there exists a real number $b$ (say) 
such that one has $0\leq \left( x\phi'(x)\right)^2\leq b$ almost everywhere in $[0,1]$. 
It follows that we have $\int_0^1 \left( x\phi'(x)\right)^2 dx \leq \int_0^1 b dx = b < \infty$. 
The measurable function $x\mapsto x\phi'(x)$ is, therefore, square integrable on $[0,1]$. 
By the bounds of Theorems~4.6 and~4.9, we have also 
\begin{align*} 
\int_0^1 \left( x\phi'(x)\right)^2 dx 
 &=O\left( \lambda^6 \log^4 |\lambda|\right) \cdot\int_0^{\frac{1}{|\lambda | \log |\lambda |}} dx 
+ O\left( \lambda^4 \log^2 |\lambda|\right) \cdot\int_{\frac{1}{|\lambda | \log |\lambda |}}^1 \frac{dx}{x^2} \\ 
 &= O\left( \lambda^6 \log^4 |\lambda|\right) \cdot \left( |\lambda | \log |\lambda | \right)^{-1} 
+ O\left( \lambda^4 \log^2 |\lambda|\right) \cdot |\lambda | \log |\lambda | , 
\end{align*} 
and so we obtain the last part of the theorem. 
\end{proof} 
\begin{definitions} 
We put now:  
\begin{equation}\label{DefQ(x)} 
Q(x) := x \phi'(x) \quad \text{($0\leq x\leq 1$ and $\phi'(x)$ is defined and finite)} 
\end{equation} 
and 
\begin{equation}\label{DefP(x)} 
P(x) := -\lambda \int_0^1 K(x,y) Q(y) dy \quad \text{($0\leq x\leq 1$).} 
\end{equation}
\end{definitions}  
Note that it follows immediately from \eqref{DefQ(x)} and Theorem~4.11 that 
the function $Q(x)$ is both measurable and square integrable on $[0,1]$. 
The same is true, when $x\in [0,1]$ is given, of the 
function $y\mapsto K(x,y)$: see \eqref{HalfHilbertSchmidt}. 
It therefore follows by the Cauchy-Schwarz inequality, combined with \eqref{HSnorm} and 
theorems of Tonelli and Fubini (for which see \cite[Theorems~6.1 and~6.10]{Wheeden and Zygmund 2015}), 
that the function $P(x)$, given by \eqref{DefP(x)}, is an element 
of $L^2 \bigl( [0,1]\bigr)$. 
With this, we are able to justify our next lemma, and the further definitions that follow it. 
\begin{theorem} Let $0<x\leq 1$. Suppose that $\phi'(x)$ exists. 
Then one has 
\begin{equation*} 
x \phi'(x) + \phi(x) - \lambda\phi(1) K(x,1) = P(x) .
\end{equation*} 
\end{theorem} 
\begin{proof} 
By Lemma~4.8 and Definitions~4.12, we find that 
\begin{multline*} 
x \phi'(x) + \phi(x) - \lambda\phi(1) K(x,1) \\ 
 = \lambda \int_0^{\Delta} K(x,y) y \phi'(y) dy 
 + P(x) + O\left( \lambda^4 x^{-1} \Delta\right)  
\end{multline*} 
for all $\Delta\in (0,1)$. The theorem will therefore follow if it can be 
shown that ${\cal E}(\Delta) := \int_0^{\Delta} K(x,y) y \phi'(y) dy$ 
satisfies ${\cal E}(\Delta)\rightarrow 0$, 
in the limit as $\Delta\rightarrow 0+$. To this end, we observe that 
it is a consequence of Theorems~4.9 and~4.11, and the definition \eqref{DefK}, that 
one has 
\begin{equation*} 
\int_0^{\Delta} | K(x,y) y \phi'(y)| dy = 
O\left( |\lambda |^3 \log^2 |\lambda| \right) \cdot \int_0^{\Delta} dy 
=O\left( \Delta |\lambda |^3 \log^2 |\lambda| \right) \end{equation*} 
for $0<\Delta < 1$. We therefore have 
$\left| {\cal E}(\Delta)\right| \leq  O\left( \Delta |\lambda |^3 \log^2 |\lambda| \right)$, 
for $0<\Delta < 1$, and so may deduce that $\lim_{\Delta\rightarrow 0+} {\cal E}(\Delta) = 0$. 
The theorem follows. 
\end{proof} 
\begin{definitions} 
For $h\in{\mathbb N}$, we put: 
\begin{equation}\label{Def-a_h}
a_h := \left\langle P , \phi_h\right\rangle = \int_0^1 P(x) \phi_h (x) dx 
\end{equation} 
and 
\begin{equation}\label{Def-b_h} 
b_h :=  \int_0^1 Q(x) \phi_h (x) dx . 
\end{equation}
\end{definitions} 
\par
Since we have already found that both 
$P(x)$ and $Q(x)$ are measurable and square integrable on $[0,1]$, 
and since the same is true of all the eigenfunctions, $\phi_1(x),\phi_2(x),\phi_3(x),\ldots $, 
it therefore follows by the Cauchy Schwarz inequality that, for each $h\in{\mathbb N}$, 
both of the functions $x\mapsto P(x)\phi_h(x)$ and $x\mapsto Q(x)\phi_h(x)$ 
are integrable on $[0,1]$. Thus the integrals occurring in \eqref{Def-a_h} and \eqref{Def-b_h} exist, 
and have finite values: so we have $a_h,b_h\in{\mathbb R}$ for all $h\in{\mathbb N}$. 
\begin{lemma} 
For $h\in{\mathbb N}$, one has 
\begin{equation*}
b_h = -\left( \frac{\lambda_h}{\lambda}\right)  a_h . 
\end{equation*} 
\end{lemma} 
\begin{proof} 
Let $h\in{\mathbb N}$. As noted in \cite[Sections~3.9--3.10]{Tricomi 1957}, it follows  
by \eqref{DefEigenfunction}, \eqref{DefK}, \eqref{DefP(x)} and Fubini's theorem for double integrals, 
that one has 
\begin{align*} 
\int_0^1 Q(x) \phi_h (x) dx &= 
\int_0^1 Q(x) \lambda_h \biggl(\int_0^1 K(x,y) \phi_h (y) dy\biggr) dx \\ 
 &= \lambda_h \int_0^1  \biggl(\int_0^1  K(x,y) Q(x) dx\biggr) \phi_h (y) dy \\ 
 &= \lambda_h \int_0^1  \biggl(\int_0^1 K(y,x)  Q(x) dx\biggr) \phi_h (y)dy \\ &= 
 \lambda_h \left\langle (-\lambda)^{-1} P , \phi_h \right\rangle 
 = \lambda_h  (-\lambda)^{-1} \left\langle P , \phi_h \right\rangle . 
\end{align*}  
By this and Definitions~4.14, one has $b_h = -\lambda^{-1}\lambda_h a_h$. 
\end{proof} 
\begin{theorem}[Hilbert-Schmidt] 
The series 
\begin{equation}\label{HSseries}  
a_1 \phi_1(x) + a_2 \phi_2(x) + a_3 \phi_3(x) + \ldots 
\end{equation} 
converges both absolutely and uniformly on $[0,1]$. For $0\leq x\leq 1$, one has 
\begin{equation}\label{HSthm}
\sum_{h=1}^{\infty} a_h \phi_h (x) = P(x) . 
\end{equation} 
\end{theorem} 
\begin{proof} 
This theorem is, in essence, just one specific case of the `Hilbert-Schmidt theorem' that is 
proved in \cite[Section~3.10]{Tricomi 1957}: note, in particular, that it follows by 
virtue of Definitions~4.12 and Theorem~4.11 that the Hilbert-Schmidt theorem 
is applicable to $P(x)$. However, the Hilbert-Schmidt theorem does not quite 
show that \eqref{HSthm} holds for all $0\leq x\leq 1$: it shows only that this equality 
holds almost everywhere in the interval $[0,1]$ (regarding this, see the Remarks  
following this proof). For this reason, we give more details regarding the proof of our theorem. 
\par
We note, firstly, that the absolute and uniform convergence of the series \eqref{HSseries} 
can be established by means of the steps in \cite[Page~112, Paragraph~1]{Tricomi 1957}: 
one may, in particular, put $N=\frac{1}{2}$ there, by virtue of the case $p=2$ of 
\eqref{HalfHilbertSchmidt}. Therefore, in order to complete this proof, we need only 
show that one has $\lim_{H\rightarrow\infty} \bigl( P(x) - \sum_{h=1}^H a_h  \phi_h (x)\bigr) = 0$ 
for $0\leq x\leq 1$. 
Accordingly, we suppose now that $x\in [0,1]$. 
By \eqref{DefEigenfunction}, \eqref{DefP(x)}-\eqref{Def-b_h} and Lemma~4.15, 
we find (similarly to \cite[Page~111, Paragraph~2]{Tricomi 1957}) that one has 
\begin{equation*} 
P(x) - \sum_{h=1}^H a_h \phi_h (x) 
= -\lambda\int_0^1 \biggl( K(x,y)  - \sum_{h=1}^H \frac{\phi_h (x) \phi_h (y)}{\lambda_h}\biggr) Q(y) dy ,
\end{equation*}
for all $H=1,2,3,\ldots\ $. It therefore follows, by the Cauchy-Schwarz inequality, that, 
for all $H=1,2,3,\ldots\ $, one has:  
\begin{multline*} 
\biggl( P(x) - \sum_{h=1}^H a_h \phi_h (x) \biggr)^2 \\ 
\leq \lambda^2 
\left( \int_0^1 \biggl( K(x,y)  - \sum_{h=1}^H \frac{\phi_h (x) \phi_h (y)}{\lambda_h}\biggr)^2 dy \right) 
\left( \int_0^1 Q^2 (y) dy\right) . 
\end{multline*}
Since we have here $\int_0^1 Q^2 (y) dy < \infty$ (by Theorem~4.11), we may therefore 
deduce from the result \eqref{K(x,y)_in_mean} of Corollary~2.12 that one 
does indeed have $\lim_{H\rightarrow\infty} \bigl( P(x) - \sum_{h=1}^H a_h  \phi_h (x)\bigr) = 0$, 
as required.  
\end{proof} 
\begin{remarks} 
Since the entire latter part of the above proof is very similar indeed to 
the reasoning that can be found in \cite[Page~111, Paragraph~2]{Tricomi 1957}, 
we should point out that, where we have appealed to our result \eqref{K(x,y)_in_mean}, 
Tricomi relies instead upon the result 
\begin{equation} \label{l.i.m.}
\lim_{H\rightarrow\infty} \int_0^1\int_0^1 \left( K(x,y) - 
\sum_{h=1}^H \frac{\phi_h (x) \phi_h(y)}{\lambda_h} \right)^2 dx dy = 0 , 
\end{equation} 
stated (in other notation) in \cite[Section~3.9, Equation~(3)]{Tricomi 1957}. 
By itself, this latter result implies only that, almost everywhere in $[0,1]$, 
one has $\lim_{H\rightarrow\infty} \int_0^1 \bigl( K(x,y) - 
\sum_{h=1}^H \lambda_h^{-1} \phi_h (x) \phi_h(y) \bigr)^2 dy = 0$:   
whereas we know, by \eqref{K(x,y)_in_mean}, 
that this equality holds for all $x\in [0,1]$. This (we hope) explains   
our earlier assertion to the effect that 
the Hilbert-Schmidt theorem proved in \cite[Section~3.10]{Tricomi 1957} 
does not (by itself) show that the equality \eqref{HSthm}  
holds for all $x\in [0,1]$.  
\end{remarks} 
\begin{corollary} 
The function $P(x)$ is continuous on $[0,1]$. In particular, one has 
$P(x)\rightarrow P(0)=0$, in the limit as $x\rightarrow 0+$. 
\end{corollary} 
\begin{proof} 
Each term of the series \eqref{HSseries} is 
(by Theorem~2.10) a continuous function on $[0,1]$. Therefore it follows, given  
the fact of the uniform convergence of this series (noted in Theorem~4.16), that this series converges 
(pointwise) to a sum that is a continuous function on $[0,1]$. By Theorem~4.16 (again), 
the sum in question is identically equal to $P(x)$, and so $P(x)$ is continuous on $[0,1]$. 
In order to complete the proof, we observe that, by the definitions \eqref{DefK} and 
\eqref{DefP(x)}, one has $P(0) = -\lambda \int_0^1 K(0,y)Q(y) dy = -\lambda \int_0^1 0 dy = 0$. 
\end{proof} 
\begin{remarks} 
In view of the above corollary, Theorem~2.10 and the definition of $K(x,y)$ in \eqref{DefK}, 
one can observe that the discontinuities of $\phi' (x)$ 
(for which see Theorem~4.1 and Corollaries~4.2 and~4.3) 
are fully accounted for by the presence, in the result of Theorem~4.13, of the term 
$\lambda \phi (1) K(x,1)$. 
\end{remarks} 
\begin{corollary}
The series $a_1 \phi_1(x) + a_2 \phi_2(x) + a_3 \phi_3(x) + \ldots $ 
is `convergent in the mean' to the function $P(x)$, in that one has  
\begin{equation}\label{P(x)_in_mean} 
\lim_{H\rightarrow\infty}\int_0^1 \biggl( P(x) - \sum_{h=1}^H a_h \phi_h(x)\biggr)^2 dx = 0 . 
\end{equation} 
The series $b_1 \phi_1(x) + b_2 \phi_2(x) + b_3 \phi_3(x) + \ldots $ converges 
in the mean to $Q(x)$: one has 
\begin{equation}\label{Q(x)_in_mean} 
\lim_{H\rightarrow\infty}\int_0^1 \biggl( Q(x) - \sum_{h=1}^H b_h \phi_h(x)\biggr)^2 dx = 0 . 
\end{equation}
One has, moreover, 
\begin{equation}\label{a_h,b_h-formulae}
-\left( \frac{\lambda_h}{\lambda}\right) \left( 1 + \frac{\lambda_h}{\lambda}\right) a_h 
= \left( 1 + \frac{\lambda_h}{\lambda}\right) b_h 
= \phi(1) \phi_h (1) - \left\langle \phi , \phi_h\right\rangle 
\quad \text{($h\in{\mathbb N}$).}   
\end{equation} 
\end{corollary}
\begin{proof}
For $H\in{\mathbb N}$, the integral 
$\int_0^1 \bigl( P(x) - \sum_{h=1}^H a_h \phi_h(x)\bigr)^2 dx \geq 0$ 
is less than or equal to the square of 
$\sup_{0\leq x\leq 1} \bigl| P(x) - \sum_{h=1}^H a_h \phi_h(x)\bigr|$. 
Since we know also (by Theorem~4.16) that 
$\sup_{0\leq x\leq 1} \bigl| P(x) - \sum_{h=1}^H a_h \phi_h(x)\bigr|\rightarrow 0$, as $H\rightarrow\infty$, 
we therefore can deduce that \eqref{P(x)_in_mean} holds. 
\par 
We now put, for each $h\in{\mathbb N}$, 
\begin{equation*} 
c_h := a_h + \frac{\lambda \phi(1) \phi_h (1)}{\lambda_h} - \left\langle \phi , \phi_h\right\rangle . 
\end{equation*} 
Since it is assumed that $\phi(x)$ is the eigenfunction $\phi_j(x)$, 
we have here that $\left\langle \phi , \phi_h\right\rangle$ equals $1$ if $h=j$, and 
is otherwise equal to $0$. Let $H\geq j$ be a positive integer. 
Then, by \eqref{DefQ(x)} and Theorem~4.13, one has 
\begin{align*} 
Q(x) - \sum_{h=1}^H c_h \phi_h (x) 
 &= P(x) - \sum_{h=1}^H a_h \phi_h (x)  \\ 
 &\phantom{{=}} + \lambda \phi(1) \biggl( K(x,1) - \sum_{h=1}^H \frac{\phi_h(x) \phi_h(1)}{\lambda_h} \biggr) ,  
\end{align*}
for all $x\in (0,1]$ such that $\phi'(x)$ exists. 
Given this, together with Theorem~4.11 and \eqref{DefQ(x)}, we find 
(via an application of the Cauchy-Schwarz inequality)  
that one has 
\begin{multline*} 
\int_0^1 \biggl( Q(x) - \sum_{h=1}^H c_h \phi_h (x) \biggr)^2 dx \\ 
 \leq \left( 1 + \lambda^2 \phi^2 (1)\right) 
 \Biggl( \int_0^1  \biggl( P(x) - \sum_{h=1}^H a_h \phi_h (x)\biggr)^2 dx \\ 
+ \int_0^1 \biggl( K(x,1) - \sum_{h=1}^H \frac{\phi_h(x) \phi_h(1)}{\lambda_h} \biggr)^2 dx 
\Biggr) .  
\end{multline*}
Therefore it follows, by \eqref{P(x)_in_mean}, \eqref{K(x,y)_in_mean} and the symmetry of 
the kernel $K$, that one has 
\begin{equation}\label{ProofI1} 
\int_0^1 \biggl( Q(x) - \sum_{h=1}^H c_h \phi_h (x) \biggr)^2 dx \rightarrow 0 , 
\quad \text{as $H\rightarrow\infty$.} 
\end{equation} 
It is, at the same time, a consequence of Theorem~4.11, \eqref{DefQ(x)}, \eqref{Def-b_h}  
and the orthonormality of $\phi_1(x),\phi_2(x),\phi_3(x),\ldots $ , 
that one has 
\begin{multline*} 
\int_0^1 \biggl( Q(x) - \sum_{h=1}^H c_h \phi_h (x) \biggr)^2 dx \\ 
\begin{aligned} 
 &= \int_0^1 \biggl( Q(x) - \sum_{h=1}^H b_h \phi_h (x) \biggr)^2 dx  
+ \sum_{h=1}^H \left( c_h - b_h\right)^2 \\ 
 &\geq \sum_{h=1}^H \left( c_h - b_h\right)^2 \geq \left( c_k - b_k\right)^2 
\quad \text{when $1\leq k\leq H$} 
\end{aligned} 
\end{multline*}
(see \cite[Section~3.2]{Tricomi 1957} regarding this). 
By this and \eqref{ProofI1}, it is necessarily the case that one has 
$c_k = b_k$ for all $k\in{\mathbb N}$. We can therefore deduce from \eqref{ProofI1}  
that \eqref{Q(x)_in_mean} holds. 
\par
Let $h\in{\mathbb N}$. Recalling the definition of $c_h$, and also Lemma~4.15, 
we have now that 
$-\lambda^{-1}\lambda_h a_h = b_h = c_h 
:= a_h + \lambda \lambda_h^{-1} \phi(1) \phi_h (1) - \left\langle \phi , \phi_h\right\rangle$. 
The equations in \eqref{a_h,b_h-formulae} follow from this 
(given that $\left\langle \phi , \phi_h\right\rangle = 0$ whenever $\lambda_h \neq \lambda$). 
\end{proof} 

\medskip 

The next corollary includes a result involving the real constant $K_2 (1,1)$.  
By \eqref{K2xxactly} and \eqref{DefK}, this constant is the number  
$\log (2\pi) - \frac74 = 0{\cdot}087877\ldots\ $. 

\medskip 

\begin{corollary}[Parseval identities] 
One has 
\begin{equation}\label{ParsevalForPandQ} 
\sum_{h=1}^{\infty} a_h^2 = \| P \|^2 \in{\mathbb R} , 
\quad 
\sum_{h=1}^{\infty} b_h^2 = \int_0^1 Q^2(x) dx \in{\mathbb R} 
\end{equation} 
and 
\begin{equation}\label{ParsevalWithTapering}
\sum_{h=1}^{\infty} \left( 1 + \frac{\lambda_h}{\lambda} \right)^2 a_h^2 
= K_2(1,1) \lambda^2 \phi^2(1) - 2 \phi^2 (1) + 1 . 
\end{equation}
\end{corollary} 
\begin{proof} 
We have seen that $P(x)$ is measurable and square integrable on $[0,1]$. 
Therefore, given \eqref{Def-a_h} and the orthonormality of $\phi_1(x),\phi_2(x),\phi_3(x),\ldots $ , 
it follows (see \cite[Section~3.2]{Tricomi 1957}) that a 
necessary and sufficient condition for \eqref{P(x)_in_mean} to hold 
is that one has $\sum_{h=1}^{\infty} a_h^2 = \int_0^1 P^2 (x) dx$. 
Thus, since we showed already (in Corollary~4.18)  that \eqref{P(x)_in_mean} does hold, 
and since $\int_0^1 P^2 (x) dx = \| P \|^2\in{\mathbb R}$ (by virtue of 
$P(x)$ being square integrable on $[0,1]$), we 
must have  $\sum_{h=1}^{\infty} a_h^2 = \int_0^1 P^2 (x) dx = \| P \|^2\in{\mathbb R}$. 
This proves the first part of \eqref{ParsevalForPandQ}: given 
\eqref{Q(x)_in_mean}, \eqref{Def-b_h}, \eqref{DefQ(x)} and Theorem~4.11, 
one can give a similar proof of the other part. 
\par 
By \eqref{P(x)_in_mean} and \eqref{Q(x)_in_mean},   
the series $\left( b_1 - a_1\right)\phi_1(x) + 
\left( b_2 - a_2\right)\phi_2(x) +  
\left( b_3 - a_3\right) \\ \phi_3(x) + \ldots $ 
is convergent in the mean to the function $x\mapsto Q(x) - P(x)$, 
which (by \eqref{DefQ(x)}, Theorem~4.13 and Corollary~4.3) is identical almost 
everywhere in $[0,1]$ to the function $x\mapsto \lambda \phi(1) K(x,1) - \phi(x)$. 
Since we have also $b_h - a_h   
= - \left( 1 + \lambda^{-1} \lambda_h\right) a_h$ for all $h\in{\mathbb N}$ 
(by virtue of Lemma~4.15), it therefore follows (similarly to how we were able to deduce  
the equalities in \eqref{ParsevalForPandQ}) that one must have 
both 
\begin{equation*} 
\int_0^1 \left( \lambda\phi(1) K(x,1) - \phi(x)\right) \phi_h (x) dx 
= - \left( 1 + \frac{\lambda_h}{\lambda} \right) a_h 
\quad\text{($h\in{\Bbb N}$)} 
\end{equation*} 
and the corresponding Parseval identity: 
\begin{equation*} 
\sum_{h=1}^{\infty} \left( 1 + \frac{\lambda_h}{\lambda} \right)^2 a_h^2 
= \int_0^1 \left( \lambda\phi(1) K(x,1) - \phi(x)\right)^2 dx . 
\end{equation*} 
Since $K$ is a symmetric kernel, the last integral above may be evaluated by 
expansion of the integrand, followed by term by term integration and the 
application of \eqref{K2symmetrically}, \eqref{DefEigenfunction} 
and the orthonormality of $\phi_1(x),\phi_2(x),\phi_3(x),\ldots $ : 
we thereby obtain the result stated in \eqref{ParsevalWithTapering}. 
\end{proof} 

\section{Asymptotics as $x\rightarrow 0+$} 

Throughout this section we assume, as in the preceding section, 
that $\phi$ is one of the eigenfunctions in the sequence $\phi_1,\phi_2,\phi_3,\ldots\ $, 
and that $\lambda$ is the corresponding eigenvalue of the kernel $K$. 

\begin{lemma} 
Let $x\in (0,1]$ be such that $\phi'(x)$ exists. Then one has 
\begin{equation*}  
x\phi'(x) = 
-\lambda \phi (1) \widetilde B_1 \left( \frac{1}{x}\right) 
-\lambda^2 \phi(1) K_2 (x , 1) 
+\lambda^2 \int_0^1 K_2 (x , z) z\phi'(z) dz\;. 
\end{equation*} 
\end{lemma} 

\begin{proof} By Theorem~4.13 and Definitions~4.12, followed by 
\eqref{DefEigenfunction} and Definitions~2.1, we have: 
\begin{multline*} 
x\phi'(x) = -\phi(x) + \lambda \phi(1) K(x , 1) -\lambda\int_0^1 K(x,y) y\phi'(y) dy \\ 
\begin{aligned} 
 &= -\phi(x) + \lambda \phi(1) K(x , 1) \\ 
 &\phantom{{=}}\ \, - \lambda\int_0^1 K(x,y) 
 \left( -\phi(y) + \lambda \phi(1) K(y , 1) -\lambda\int_0^1 K(y , z) z\phi'(z) dz \right)  dy \\ 
 &= -\phi(x) + \lambda \phi(1) K(x , 1) \\ 
 &\phantom{{=}}\ \, + \phi (x) - \lambda^2 \phi(1) K_2 (x , 1) 
 + \lambda^2 \int_0^1 K(x,y) \left( \int_0^1 K(y , z) z\phi'(z) dz \right)  dy\;. 
\end{aligned}
\end{multline*}
The lemma therefore follows by observing that $K(x , 1) = - \widetilde B_1 (1/x)$, by \eqref{KtoTildeB1}, 
that $-\phi(x) +\phi(x) =0$,  
and that, by virtue of Theorem~4.11, an application of Fubini's Theorem 
\cite[Theorem~6.1]{Wheeden and Zygmund 2015} gives: 
\begin{multline*} 
\int_0^1 K(x,y) \left( \int_0^1 K(y , z) z\phi'(z) dz \right)  dy \\ 
 = \int_0^1 \left( \int_0^1 K(x,y)  K(y , z)  dy \right) z\phi'(z) dz 
 = \int_0^1 K_2 (x , z) z\phi'(z) dz 
\end{multline*} 
(the last equality following from Definitions~2.1). 
\end{proof} 

\begin{definitions} 
We define 
\begin{equation}\label{DefI_0} 
I_0 (x , y) = \int_0^x K_2 (z ,  y) dz 
\qquad\text{($0\leq x,y\leq 1$)} . 
\end{equation} 
For $0<x,y\leq 1$ and $0\leq w\leq 1$, we put: 
\begin{equation}\label{DefI} 
I(x , y ; w) = \int_x^y \frac{K_2 (z , w) dz}{z^2}\;. 
\end{equation} 
\end{definitions} 

\begin{lemma} 
One has 
\begin{equation} 
I_0 (x , y) \ll \left( 1 + \log\frac{1}{y}\right) y 
\qquad\text{($0<x,y\leq 1$)} . 
\end{equation}
\end{lemma} 

\begin{proof} 
Let $x,y\in(0,1]$. By \eqref{DefI_0} and \eqref{K2BoundOffDiagonal}, we find that  
\begin{align*} 
\left| I_0 (x , y)\right|\leq \int_0^1 \left| K_2 (z , y)\right| dz  
 &=\int_0^y O\left( \frac{z}{y}\right) dz + \int_y^1 O\left( \frac{y}{z}\right) dz \\
 &= O\left( y\right) + O\left( y\log\frac{1}{y}\right) \;,
\end{align*} 
as required. 
\end{proof} 

\begin{theorem} 
Let $x\in (0,1]$ be such that $\phi' (x)$ exists. Then one has 
\begin{equation*} 
x \phi' (x) = - \lambda \phi(1) \widetilde B_1 \left( \frac{1}{x}\right) 
+ O\left( |\lambda|^5 \left( \log |\lambda|\right)^2 \left( 1 + \log\frac{1}{x}\right) x\right) \;. 
\end{equation*} 
\end{theorem} 

\begin{proof} 
In view the bounds \eqref{PhiBoundedUniformly} and \eqref{SpectrumBound}, 
the theorem will follow from Lemma~5.1, once it is shown that one has both  
$K_2 (x , 1)\ll  x$ and 
\begin{equation*} 
\int_0^1 K_2 (x , z) z\phi'(z) dz 
\ll |\lambda|^3 \left( \log |\lambda|\right)^2 \left( 1 + \log\frac{1}{x}\right) x\;. 
\end{equation*}  
The first of these two estimates is contained in \eqref{K2BoundOffDiagonal}. 
The other follows by noting that one has 
$\int_0^1 K_2 (x , z) z\phi'(z) dz 
\ll |\lambda|^3 \left( \log |\lambda|\right)^2 \int_0^1 \left| K_2 (z , x)\right| dz$ 
(by Theorems~4.9 and 4.11, and the symmetry of $K_2$) and recalling that, 
in our proof of Lemma~5.3, we found that  
$\int_0^1 \left| K_2 (z , x)\right| dz =O\left( x + x\log(1/x)\right)$. 
\end{proof} 

\begin{remarks} 
Since $\int_0^1 \left| 1 + \log (1/x)\right| dx = \int_0^1 \left( 1 + \log (1/x)\right) dx 
=2 < \infty$, while 
$\int_0^1 \bigl| \widetilde B_1 (1/x)\bigr| x^{-1} dx = \int_1^{\infty} \bigl| \widetilde B_1 (t)\bigr| t^{-1} dt 
\geq\sum_{n=1}^{\infty} \int_n^{n+1/3} \frac{1}{6} t^{-1} dt =\infty$, 
it is therefore a corollary of Theorems~4.11 and 5.4 that $\phi'(x)$ is Lebesgue integrable 
on $[0,1]$ if and only if $\phi(1)=0$. Thus if $\phi'(x)$ is Lebesgue integrable on $[0,1]$  
then, by \eqref{phi(1)asIntegral}, one has $\int_0^1 \phi' (x) dx = \phi(1) = 0$. 
\end{remarks} 

\begin{lemma} 
For $0<x,y\leq 1$, one has: 
\begin{equation*} 
I_0 (x , y) = - {\textstyle\frac{1}{6}} \left( \frac{x}{y}\right)^{\!\!3 }
\sum_{m>\frac{1}{y}} \frac{\widetilde B_3 \left( \frac{my}{x}\right)}{m^3} 
+O\left( \frac{(x+y)x^3}{y}\right) 
\ll \frac{x^3}{y}\;, 
\end{equation*} 
where $\widetilde B_3 (t) := \{ t\}^3 - {\textstyle\frac{3}{2}} \{ t\}^2 + {\textstyle\frac{1}{2}} \{ t\}\,$ 
(the third periodic Bernouilli function). 
\end{lemma} 

\begin{proof}
Let $0<x,y\leq 1$. Define functions $f_1,\ldots ,f_4$ on $[0,1]$ by putting 
$f_j (0) = 0$ ($j=1,\ldots ,4$) and  
\begin{equation*} 
f_1(z) = {\textstyle\frac{1}{2}} \widetilde B_1 \left( \frac{1}{y}\right) z \widetilde B_2 \left( \frac{1}{z}\right) \;,  
\qquad 
f_2(z) = \frac{1}{z} \int_{\frac{1}{z}}^{\infty} \widetilde B_2 (t) 
\widetilde B_1 \left( \frac{t z}{y}\right) \frac{dt}{t^3} 
\end{equation*}
\begin{equation*} 
f_3(z) = \frac{1}{2y} \int_{\frac{1}{z}}^{\infty} \widetilde B_2 (t) \frac{dt}{t^2} 
\quad\text{and}\quad
f_4(z) = \frac{z}{2y^2} 
\sum_{m > \frac{1}{y}} \frac{\widetilde B_2 \left( \frac{my}{z}\right)}{m^2}\;, 
\end{equation*} 
for $0<z\leq 1$. 
If these four functions are Lebesgue integrable on $[0,1]$  
then, by \eqref{DefI_0} and Lemma~2.3, we will have: 
\begin{equation}\label{Proof-s7A*1} 
I_0 (x,y) = \sum_{j=1}^4 (-1)^j \int_0^x f_j(z) dz = \sum_{j=1}^4 (-1)^j I_j\quad\text{(say)} . 
\end{equation} 
\par 
The function $\widetilde B_2$ is continuous and periodic on ${\mathbb R}$, and is therefore bounded.  
It follows that $f_1$ and $f_3$ are continuous on $(0,1]$. For $0<z\leq 1$, one has 
$f_1(z) =O(z)$ and $f_3 (z) = \frac{1}{2} y^{-1} \int_{1/x}^{\infty} O\left( t^{-2}\right) dt  = O(x/y)$, 
and so $f_1(z)$ and $f_3(z)$ are also continuous at the point $z=0$. 
Thus $f_1$ and $f_3$ are continuous on $[0,1]$. Similarly, each 
term in the infinite series $\sum_{m>1/y} m^{-2} z \widetilde B_2 (my/z)$ is 
continuous (as a function of $z$) on the interval $(0,1]$, and tends 
to the limit~$0$ as $z\rightarrow 0+$. Since this series 
sums to $2 y^2 f_4(z)$, and converges uniformly for $0 < z\leq 1$, we may 
conclude that the function $f_4$ is continuous on $[0,1]$. 
By Lemma~2.3, Definitions~2.1 and \eqref{DefK}, 
we have $\sum_{j=1}^4 (-1)^j f_j(z) = K_2(z,y)\,$ ($0\leq z\leq 1$). 
Therefore, since $f_1$, $f_3$ and $f_4$ are continuous on $[0,1]$, 
and since the same is true of the function $z\mapsto K_2 (z,y)\,$ 
(see Corollary~2.8), we deduce that the function $f_2$ is continuous on $[0,1]$. 
Thus the functions $f_1,\ldots ,f_4$ are integrable on $[0,1]$,   
since they are continuous on this interval. 
We therefore do have \eqref{Proof-s7A*1}. 
\par 
We shall complete the proof of the lemma by estimating 
the integrals $I_1,\ldots ,I_4$. 
We note, firstly, that 
\begin{equation*} 
I_1 = {\textstyle\frac{1}{2}} \widetilde B_1 \left( \frac{1}{y}\right) 
\int_0^x  z \widetilde B_2 \left( \frac{1}{z}\right) dz = 
{\textstyle\frac{1}{2}} \widetilde B_1 \left( \frac{1}{y}\right)  
\int_{\frac{1}{x}}^{\infty} t^{-3} \widetilde B_2 (t) dt  
\end{equation*} 
(by means of the substitution $z=1/t$). 
Since $\widetilde B_3$ is bounded, and 
satisfies $\frac{d}{dt} \widetilde B_3 (t) = 3 \widetilde B_2 (t)\,$ ($t\in{\mathbb R}$), 
we find (through integration by parts) that 
\begin{align*} 
\int_{\frac{1}{x}}^{\infty} t^{-3} \widetilde B_2 (t) dt 
 &= - {\textstyle\frac{1}{3}} x^3 \widetilde B_3\left( \frac{1}{x}\right) 
+ \int_{\frac{1}{x}}^{\infty} t^{-4} \widetilde B_3 (t) dt \\ 
 &= O\left( x^3\right) + \int_{\frac{1}{x}}^{\infty} O\left( t^{-4}\right) dt 
 \ll x^3\;. 
\end{align*} 
It follows, since $\bigl| \widetilde B_1 (1/y)\bigr| \leq \frac{1}{2}$, 
that we have: 
\begin{equation}\label{Proof-s7A*2}
I_1 = O\left( x^3\right) \;. 
\end{equation}
\par
Using the substitution $t=(zw)^{-1}$, we find that 
\begin{equation*} 
f_2 (z) = \int_0^1 zw \widetilde B_2 \left( \frac{1}{zw}\right) B_1 \left( \frac{1}{yw}\right) dw\;. 
\end{equation*} 
By this and Fubini's theorem, we have  
\begin{align*} 
I_2 = \int_0^x f_2 (z) dz &= 
\int_0^1 \left( \int_0^x z \widetilde B_2 \left( \frac{1}{zw}\right) dz\right) w B_1 \left( \frac{1}{yw}\right) dw \\ 
 &= \int_0^1 \left( \int_0^{wx} u \widetilde B_2 \left( \frac{1}{u}\right) du\right) 
w^{-1}  B_1 \left( \frac{1}{yw}\right) dw \;. 
\end{align*} 
Therefore, by a calculation similar to that which gave us \eqref{Proof-s7A*2},  
we obtain: 
\begin{equation} \label{Proof-s7A*3}
I_2 = \int_0^1 O\left ( (wx)^3\right) \cdot w^{-1}  B_1 \left( \frac{1}{yw}\right) dw 
=\int_0^1 O\left( x^3 w^2\right) dw \ll x^3\;. 
\end{equation} 
\par
Regarding $I_3$ (and $f_3(z)$), we note that integration by parts (twice) gives  
\begin{align*} 
\int_{\frac{1}{z}}^{\infty} \widetilde B_2 (t) \frac{dt}{t^2} 
 &= - {\textstyle\frac{1}{3}} z^2 \widetilde B_3 \left( \frac{1}{z}\right) 
 + {\textstyle\frac{2}{3}} \int_{\frac{1}{z}}^{\infty} \widetilde B_3 (t) \frac{dt}{t^3}  \\ 
 &= - {\textstyle\frac{1}{3}} z^2 \widetilde B_3 \left( \frac{1}{z}\right) 
 - {\textstyle\frac{1}{6}} z^3 \widetilde B_4 \left( \frac{1}{z}\right) 
 + {\textstyle\frac{1}{2}} \int_{\frac{1}{z}}^{\infty} \widetilde B_4 (t) \frac{dt}{t^4}  \\ 
 &= - {\textstyle\frac{1}{3}} z^2 \widetilde B_3 \left( \frac{1}{z}\right)  
 +O\left( z^3\right)  \qquad\text{($0<z\leq 1$)} ,
\end{align*} 
since the periodic Bernouilli function $\widetilde B_4$ is bounded. 
Note also that one has 
$\int_0^x z^2 \widetilde B_3 (1/z) dz = \int_{1/x}^{\infty} t^{-4} \widetilde B_3 (t) dt$ 
(by the substitution $z=1/t$), 
and so we find (similarly to the above calculation) that one has 
\begin{equation*} 
\int_0^x z^2 \widetilde B_3 \left( \frac{1}{z}\right)  dz 
= - {\textstyle\frac{1}{4}} x^4 \widetilde B_4 \left( \frac{1}{x}\right)   + O\left( x^5\right)  
\ll x^4\;.
\end{equation*} 
By the preceding observations, we have 
\begin{equation}\label{Proof-s7A*4}
2y I_3 =\int_0^x 2y f_3(z) dz 
= -{\textstyle\frac{1}{3}} \int_0^ x  z^2\widetilde B_3 \left( \frac{1}{z}\right) dz  
+ \int_0^x O\left( z^3\right) dz \ll x^4\;. 
\end{equation} 
\par 
Turning, lastly, to $I_4$, we note that, since 
the series $\sum_{m>1/y} m^{-2} z \widetilde B_2 (my/z)$ is uniformly convergent 
for $0<z\leq x$, we may integrate term-by-term to get: 
\begin{equation*} 
I_4 = \int_0^x f_4 (z) dz 
= \sum_{m > \frac{1}{y}} \frac{1}{2 y^2 m^2} \int_0^x        
z \widetilde B_2 \left( \frac{my}{z}\right) dz\;.   
\end{equation*} 
Using the substitution $z=my/t$, followed by integration by parts (twice), one finds that 
when $m>0$ one has:  
\begin{align*}
\frac{1}{y^2 m^2}\int_0^x z \widetilde B_2 \left(\frac{my}{z}\right) dz  
 &= \int_{\frac{my}{x}}^{\infty}  t^{-3} \widetilde B_2 (t) dt \\ 
 &= - \sum_{r=3}^4\frac{1}{r} \left( \frac{my}{x}\right)^{\!-r} \widetilde B_r\left( \frac{my}{x}\right)  
 + \int_{\frac{my}{x}}^{\infty}  t^{-5} \widetilde B_4 (t) dt \\ 
 &= - \frac{1}{3} \left( \frac{my}{x}\right)^{\!-3} \widetilde B_3\left( \frac{my}{x}\right)  
 + O\left( \left( \frac{my}{x}\right)^{\!-4}\right) \;. 
\end{align*} 
Thus, since $\sum_{m>1/y} m^{-4} = O\left( y^3\right)$, we get: 
\begin{equation}\label{Proof-s7A*5}
I_4 = - {\textstyle\frac{1}{6}} \left( \frac{x}{y}\right)^{\!\!3} 
\sum_{m>\frac{1}{y}} m^{-3} \widetilde B_3\left( \frac{my}{x}\right) 
+ O\left( \frac{x^4}{y}\right) \;. 
\end{equation} 
\par 
By \eqref{Proof-s7A*1}--\eqref{Proof-s7A*5}, we conclude that    
\begin{equation*} 
I_0 (x,y) = - {\textstyle\frac{1}{6}} \left( \frac{x}{y}\right)^{\!\!3} 
\sum_{m>\frac{1}{y}} m^{-3} \widetilde B_3\left( \frac{my}{x}\right) 
+ O\left( \frac{x^4}{y}\right) + O\left( x^3\right) \;. 
\end{equation*} 
The lemma follows, since $\sum_{m>1/y} m^{-3} \widetilde B_3 (my/x) \ll \sum_{m>1/y} m^{-3} \ll y^2$  
and $O\left( x^4 / y\right) + O\left( x^3\right) \ll (x+y) x^3 / y \leq 2 x^3 / y$. 
\end{proof} 

As a corollary of Lemmas~5.3 and 5.5, we obtain the following lemma 
concerning the integral $I(x,y;w)$ that we have defined in \eqref{DefI}. 

\begin{lemma} Let $0<x,y\leq 1$. Then 
\begin{equation}\label{IxywBound}
I(x,y;w)\ll \min\left\{ \frac{x+y}{w} \,,\, \left( \frac{1}{x^2} + \frac{1}{y^2}\right) 
\left( 1 + \log\frac{1}{w}\right) w\right\} 
\end{equation} 
for $0<w\leq 1$, and one has 
\begin{equation}\label{IxywL1norm}
\int_0^1 \left| I(x,y;w)\right| dw \ll \left( 1 + \log\frac{1}{xy}\right) (x + y)\;. 
\end{equation} 
\end{lemma} 

\begin{proof} 
Let $w\in (0,1]$. Given the definitions \eqref{DefI} and \eqref{DefI_0}, 
we find, using integration by parts, that one has 
\begin{equation*} 
I(x,y;w) = y^{-2} I_0 (y , w) - x^{-2} I_0 (x, w) + \int_x^y 2 z^{-3} I_0 (z , w) dz\;. 
\end{equation*} 
By Lemma~5.5 each term of form $I_0 (u, w)$ occurring in the last equation 
is of size $O(u^3 / w)$. Thus we find that 
$I(x,y;w) = O(y / w) - O( x / w) + \int_x^y O(1 / w) dz \ll (x+y)/w$.  
By using Lemma~5.3, in place of Lemma~5.5, one obtains the different estimate: 
\begin{align*} 
I(x,y;w) &= O\left( y^{-2} \left( 1 + \log\frac{1}{w}\right) w \right) 
- O\left( x^{-2}\left( 1 + \log\frac{1}{w}\right) w \right) \\ 
 &\phantom{{=}} \ \, + \int_x^y O\left( z^{-3} \left( 1 + \log\frac{1}{w}\right) w \right) dz \\
&\ll \left( \frac{1}{y^2} + \frac{1}{x^2}\right) \left( 1 + \log\frac{1}{w}\right) w\;.     
\end{align*} 
This completes the proof of \eqref{IxywBound}.
\par 
We now put 
\begin{equation*} 
\delta = \frac{xy}{\sqrt{x+y}} \;, 
\end{equation*} 
so that $0<\delta <x\sqrt{y}\leq 1$ and $1/\delta < 2/(xy)$. 
By \eqref{IxywBound}, we have 
\begin{align*} 
\int_0^1 \left| I(x,y;w)\right| dw 
 &= \int_0^{\delta} \left| I(x,y;w)\right| dw + \int_{\delta}^1 \left| I(x,y;w)\right| dw \\ 
 &\ll \left( \frac{1}{x^2} + \frac{1}{y^2}\right) \int_0^{\delta} \left( 1 + \log\frac{1}{w}\right) w dw 
+ (x + y) \int_{\delta}^1 \frac{dw}{w} \\ 
&= \left( \frac{y^2 + x^2}{x^2 y^2} \right) 
\delta^2 \left( {\textstyle\frac{3}{4}} + {\textstyle\frac{1}{2}}\log\frac{1}{\delta}\right)  
+ (x + y)\log\frac{1}{\delta} \\ 
 &\ll \left( \frac{x^2 + y^2}{x+y} + x + y\right) \left( 1 + \log\frac{1}{xy}\right) \;. 
\end{align*} 
The result \eqref{IxywL1norm} follows. 
\end{proof} 

\begin{definitions} 
We define 
\begin{equation}\label{DefPhi_0}
\Phi_0 (x) = \int_0^x \phi(y) dy\qquad\text{($0\leq x\leq 1$)} . 
\end{equation} 
For $0<x,y\leq 1$ and $\sigma < 3$, we put 
\begin{equation}\label{DefPhixysigma}
\Phi (x , y ; \sigma) = \int_x^y \frac{\phi (z) dz}{z^{\sigma}} \;. 
\end{equation} 
\end{definitions} 

\begin{theorem} 
One has 
\begin{equation*} 
\Phi_0 (x) \ll |\lambda|^3 x^3 \min\left\{ \lambda^2 \,,\, 1 + \log\frac{1}{x}\right\} 
\qquad\text{($0<x\leq 1$)} . 
\end{equation*} 
\end{theorem} 

\begin{proof} 
Let $0<x\leq 1$. By \eqref{DefPhi_0}, \eqref{K2Eigenfunction} and Fubini's theorem for double integrals, 
it follows that one has 
\begin{equation}\label{Proof-s7B1} 
\Phi_0 (x) = \int_0^x \left( \lambda^2 \int_0^1 K_2 (y , z) \phi(z) dz\right) dy 
=\lambda^2 \int_0^1  I_0 (x , z) \phi(z) dz\;,
\end{equation} 
where $I_0 (x , z)$ is given by \eqref{DefI_0}. 
By \eqref{Proof-s7B1}, Theorem~3.2 and Lemma~5.5, we have  
\begin{equation}\label{Proof-s7B2}
\Phi_0 (x) =\lambda^2 \int_0^1 O\left( x^3 z^{-1}\right) \cdot 
O\left( |\lambda|^3 z\right) dz 
\ll |\lambda|^5 x^3\;. 
\end{equation} 
\par 
We now put: 
\begin{equation*} 
\delta = x^{3/2}\;, 
\end{equation*}
so that $0<\delta\leq 1$. We observe that, by Lemmas~5.3 and 5.5, 
one has 
\begin{align*} 
\int_0^1 \left| I_0 (x,z)\right| dz &= \int_0^{\delta} \left| I_0 (x,z)\right| dz 
+ \int_{\delta}^1 \left| I_0 (x,z)\right| dz \\
 &=\int_0^{\delta} \left( 1 + \log\frac{1}{z}\right) z dz 
+ \int_{\delta}^1 x^3 z^{-1} dz \\ 
 &= \left( {\textstyle\frac{3}{4}} + {\textstyle\frac{1}{2}}\log\frac{1}{\delta}\right) \delta^2 
+ x^3 \log\frac{1}{\delta} \ll \left( 1 + \log\frac{1}{x}\right) x^3 \;. 
\end{align*} 
By this, \eqref{PhiBoundedUniformly} and \eqref{Proof-s7B1}, it follows that 
\begin{equation}\label{Proof-s7B3} 
\Phi_0 (x) \ll |\lambda|^3 \int_0^1 \left| I_0 (x,z)\right| dz 
\ll |\lambda|^3 \left( 1 + \log\frac{1}{x}\right) x^3\;.  
\end{equation} 
The combination of \eqref{Proof-s7B2} and \eqref{Proof-s7B3} implies the theorem. 
\end{proof} 

\begin{corollary} 
Let $\sigma < 3$. Then one has 
\begin{equation*} 
\Phi (x,y;\sigma) \ll_{\sigma} |\lambda|^3 (x+y)^{3-\sigma} 
\min\left\{ \lambda^2 \,,\, 1 + \log\frac{1}{xy} \right\} 
\qquad\text{($0<x,y\leq 1$)} . 
\end{equation*} 
\end{corollary} 

\begin{proof} 
Let $x,y\in (0,1]$. It follows from \eqref{DefPhixysigma}, by integration by parts, 
that 
\begin{equation*} 
\Phi(x,y;\sigma) 
= y^{-\sigma} \Phi_0 (y) - x^{-\sigma} \Phi_0 (x) + \sigma \int_x^y z^{-\sigma - 1} \Phi_0 (z) dz\;, 
\end{equation*} 
where $\Phi_0 (z)$ is given by \eqref{DefPhi_0}. By this and Theorem~5.8, we have both 
\begin{align*} 
\Phi(x,y;\sigma) &\ll |\lambda|^5 \left( y^{3-\sigma} + x^{3-\sigma} 
+\left| \sigma \int_x^y z^{2-\sigma} dz\right|\right) \\
 &= |\lambda|^5 \left( y^{3-\sigma} + x^{3-\sigma}  
+ \frac{|\sigma|}{(3-\sigma)} \left| y^{3-\sigma} - x^{3-\sigma}\right|\right) \\ 
 &\ll_{\sigma} |\lambda|^5 (x+y)^{3-\sigma}   
\end{align*} 
and 
\begin{equation*} 
\Phi(x,y;\sigma) \ll_{\sigma} |\lambda|^5 (x+y)^{3-\sigma}  
\cdot \lambda^{-2} \left( 1 + \log\frac{1}{x} + \log\frac{1}{y}\right) \;. 
\end{equation*} 
The last two estimates imply the corollary. 
\end{proof} 

\begin{lemma} 
One has 
\begin{align*} 
\frac{\phi(y)}{y} - \frac{\phi(x)}{x} &= {\textstyle\frac{1}{2}} \lambda \phi(1) 
\left( \widetilde B_2 \left( \frac{1}{y}\right) - \widetilde B_2 \left( \frac{1}{x}\right)\right) \\ 
 &\phantom{{=}} \ \, + O\left( |\lambda|^5 \left( \log |\lambda|\right)^2 
(x+y)\left( 1 + \log\frac{1}{xy}\right)\right)  
\end{align*}
for $0<x,y\leq 1$. 
\end{lemma} 

\begin{proof} 
Let $0<x<y\leq 1$. 
Recalling our Remarks following Lemma~3.4, we 
note (in particular) that the function $\phi$ satisfies a uniform Lipschitz condition 
of order $1$ on the interval $[x,y]$, and is (therefore) absolutely continuous on 
this interval. Since the same is true of the function $z\mapsto z^{-1}$, it 
follows that the function $z\mapsto z^{-1} \phi(z)$ satisfies a uniform Lipschitz 
condition of order $1$ on $[x,y]$, and so (like $\phi$) is absolutely 
continuous on this interval. Therefore  
\begin{align*} 
\frac{\phi(y)}{y} - \frac{\phi(x)}{x} 
 &= \int_x^y \left( \frac{d}{dz}\left( \frac{\phi(z)}{z}\right) \right) dz \\ 
 &= \int_x^y \left( \frac{\phi'(z)}{z} - \frac{\phi(z)}{z^2} \right) dz 
= \int_x^y \frac{\phi'(z) dz}{z} - \int_x^y \frac{\phi(z)dz}{z^2} dz\;. 
\end{align*} 
The last of the above integrals is $\Phi (x,y;2)\,$ (see the Definitions~5.7). 
Thus, by Corollary~5.9, we have  
\begin{equation}\label{Proof-s7C1}
\frac{\phi(y)}{y} - \frac{\phi(x)}{x} 
= \int_x^y \frac{\phi'(z) dz}{z} + O\left( |\lambda|^5 y\right) \;. 
\end{equation} 
\par 
By Lemma~5.1 and a simple substitution, one has 
\begin{align}\label{Proof-s7C2} 
\int_x^y \frac{\phi'(z) dz}{z} 
 &= \int_x^y  \biggl( -\lambda \phi (1) z^{-2} \widetilde B_1 \left( \frac{1}{z}\right) 
-\lambda^2 \phi(1) z^{-2} K_2 (z , 1) \nonumber\\ 
 &\phantom{{=}} \ \, +\lambda^2 z^{-2} \int_0^1 K_2 (z , w) w\phi'(w) dw \biggr) dz \nonumber\\ 
 &= \lambda \phi (1) \int_{\frac{1}{x}}^{\frac{1}{y}} \widetilde B_1 (t) dt 
 - \lambda^2 \phi(1) I(x,y;1) 
+ \lambda^2 J(x,y)\;,  
\end{align} 
where 
\begin{equation*} 
J(x,y) := \int_x^y \left( \int_0^1 z^{-2} K_2 (z , w) w\phi'(w) dw \right) dz\;, 
\end{equation*} 
while $I(x,y;w)$ is as defined in \eqref{DefI}. In view of Theorems~2.7 and~4.11, 
it follows by Fubini's theorem for double integrals that we have here: 
$J(x,y) = \int_0^1 I(x,y;w) \cdot w\phi'(w) dw$. 
By this, together with Theorem~4.9 and the estimate 
\eqref{IxywL1norm} of Lemma~5.6, we find that 
\begin{align}\label{Proof-s7C3}
J(x,y) &\ll |\lambda|^3 \left(\log |\lambda|\right)^2 \int_0^1 \left| I(x,y;w)\right| dw  \nonumber\\ 
 &\ll |\lambda|^3 \left(\log |\lambda|\right)^2\left( 1 + \log\frac{1}{x}\right) y\;. 
\end{align} 
\par 
By \eqref{PhiBoundedUniformly} and the estimate \eqref{IxywBound} of Lemma~5.6, 
we have $\phi(1) I(x,y;1)\ll |\lambda| y$. 
This, combined with \eqref{Proof-s7C1}, \eqref{Proof-s7C2} and \eqref{Proof-s7C3}, 
shows that  
\begin{equation*} 
\frac{\phi(y)}{y} - \frac{\phi(x)}{x} 
= \lambda \phi (1) \int_{\frac{1}{x}}^{\frac{1}{y}} \widetilde B_1 (t) dt  
+ O\left( |\lambda|^5 \left( \lambda^{-2} + \log^2 |\lambda|\right)   
\left( 1 + \log\frac{1}{x}\right) y\right) \;. 
\end{equation*} 
\par 
This completes our proof of those cases of the lemma in which one has 
$0<x<y\leq 1$: for one has 
$\int_a^b \widetilde B_1 (t) dt =  \frac{1}{2} \widetilde B_2 (b) - \frac{1}{2} \widetilde B_2 (a)\,$ 
($a,b\in{\mathbb R}$), and we know (see \eqref{SpectrumBound}) that $|\lambda|\geq 2$, 
so that one has $\log^2 |\lambda| \gg 1 \gg \lambda^{-2}$. 
The cases where $0<y<x\leq 1$ follow trivially from the cases just established, 
since swapping $x$ for $y$ in the equation occurring in the statement of the lemma 
has the same effect as multiplying both sides of that equation by $-1$. 
The remaining cases of the lemma (those where $x=y$) are trivially valid. 
\end{proof} 

\begin{theorem} 
One has 
\begin{equation*} 
\frac{\phi(x)}{x} = {\textstyle\frac{1}{2}} \lambda \phi(1) \widetilde B_2 \left( \frac{1}{x}\right) 
+ O\left( |\lambda|^5 \left( \log |\lambda|\right)^2 x \left( 1 + \log\frac{1}{x}\right)\right)  
\end{equation*}
for $0<x\leq 1$. 
\end{theorem} 

\begin{proof} 
Let $0<x\leq 1$. By Lemma~5.10, we have 
\begin{multline*} 
\frac{\phi(x)}{x} - {\textstyle\frac{1}{2}} \lambda \phi(1) \widetilde B_2 \left( \frac{1}{x}\right) \\  
=  \frac{\phi(y)}{y} - {\textstyle\frac{1}{2}} \lambda \phi(1) \widetilde B_2 \left( \frac{1}{y}\right) 
+ O\left( |\lambda|^5 \left( \log |\lambda|\right)^2 
x\left( 1 + \log\frac{1}{y}\right)\right)  
\end{multline*}
for $0<y\leq x$. By multiplying both sides of the last equation by $y$, and then 
integrating (with respect to $y$) over the interval $(0,x]$, we deduce that 
\begin{align}\label{Proof-s7D1} 
{\textstyle\frac{1}{2}} x\phi (x)   
- {\textstyle\frac{1}{4}} \lambda \phi(1)x^2\widetilde B_2 \left( \frac{1}{x}\right)   
 &=\Phi_0 (x) 
- {\textstyle\frac{1}{2}} \lambda \phi(1) \int_0^x y \widetilde B_2 \left( \frac{1}{y}\right) dy \nonumber\\ 
 &\phantom{{=}}\ + O\left( |\lambda|^5 \left( \log |\lambda|\right)^2 x^3 \left( 1 + \log\frac{1}{x}\right) \right) ,     
\end{align}  
where $\Phi_0 (x)$ is given by \eqref{DefPhi_0}. 
We recall, from our treatment of the integral $I_1$ in the proof of Lemma~5.5, 
that one has  here 
$\int_0^x y \widetilde B_2 ( 1/y) dy \ll x^3$. 
By Theorem~5.8, we have $\Phi_0 (x) \ll |\lambda|^5 x^3$. 
Given \eqref{PhiBoundedUniformly} and 
\eqref{SpectrumBound}, it follows from the last two observations that the entire right-hand side 
of equation \eqref{Proof-s7D1} is of size 
$O\left( |\lambda|^5 \left( \log |\lambda|\right)^2 x^3 \left( 1 + \log\frac{1}{x}\right) \right)$. 
Thus we obtain an estimate,  
\begin{equation*} 
\left(\frac{\phi (x)}{x}  
- {\textstyle\frac{1}{2}} \lambda \phi(1)\widetilde B_2 \left( \frac{1}{x}\right)\right) 
\cdot {\textstyle\frac{1}{2}} x^2    
\ll |\lambda|^5 \left( \log |\lambda|\right)^2 x^3 \left( 1 + \log\frac{1}{x}\right) , 
\end{equation*} 
from which the required result follows. 
\end{proof} 

\begin{remarks} 
It follows from Theorem~5.11 that the righthand derivative $\phi_{+}' (0)$ is equal to $0$   
if $\phi(1)=0$, but does not exist if $\phi(1)\neq 0$. 
\end{remarks}

\end{document}